\documentclass[11pt]{article} 
\pdfoutput=1
\usepackage[utf8]{inputenc}
\usepackage[T3,T1]{fontenc}
\usepackage[english]{babel}
\makeatletter\newcommand{\note}[1]{{\textcolor{red}{[#1]}}\@latex@warning{Note: #1}}\makeatother
\usepackage{amsmath,amsfonts,amssymb,amsthm}
\usepackage{enumitem}
\usepackage{mathtools}
\usepackage{booktabs}

\usepackage{multirow}

\DeclareSymbolFont{tipa}{T3}{cmr}{m}{n}
\DeclareMathAccent{\invbreve}{\mathalpha}{tipa}{16}

\usepackage[mono=false]{libertine}
\usepackage[cmintegrals,libertine]{newtxmath}
\usepackage[cal=euler, scr=boondoxo]{mathalfa}
\usepackage{physics}

\useosf
\usepackage[a4paper,vmargin={3.5cm,3.5cm},hmargin={2.5cm,2.5cm}]{geometry}
\usepackage[font={small,sf,it}, labelfont={sf,bf}, margin=1cm]{caption}
\usepackage[colorlinks=true]{hyperref}
\usepackage{color,graphicx}
\usepackage{placeins}

\usepackage{stackrel}
\usepackage{soul}

\theoremstyle{plain}
\newtheorem{theorem}{Theorem}
\newtheorem{corollary}[theorem]{Corollary}
\newtheorem{proposition}[theorem]{Proposition}
\newtheorem{lemma}[theorem]{Lemma}
\theoremstyle{definition}
\newtheorem{definition}[theorem]{Definition}

\newtheorem{remark}[theorem]{Remark}

\newcommand{\ind}{\mathbf{1}}

\newcommand{\Z}{\mathbb{Z}}
\newcommand{\Q}{\mathbb{Q}}

\begin{document}
\title{A tree bijection for rigid quadrangulations}
\author{Bart Zonneveld}
\maketitle
\begin{abstract}
	We study the counting problem of rigid quadrangulations, recently introduced by Budd \cite{Budd_rectilinear_2025} and proven to be in bijection with colorful quadrangulations. 
    The generating function for the latter has been derived in an algebraic manner by Bousquet-Mélou and Elvey Price \cite{Bousquet-Melou_generating_2020}, which therefore also counts rigid quadrangulations. 
    In this paper we will provide a direct, bijective proof, for this generating function. 
    We will relate the rigid quadrangulations to some naturally appearing trees, decorated with certain natural data. 
    By some slight bijective manipulation of the data, we get a decorated tree, for which the generating function can be found.

    This result opens the door to better understand the geometry of random rigid quadrangulations (and maybe even of random colorful quadrangulations), by studying the corresponding decorated trees.
    These properties are relevant the formulation of UV complete JT-gravity, following the work of Ferrari \cite{Ferrari_Random_Disk_overview_2025,Ferrari_Random_Disk_lattice_2024}.
\end{abstract}

\section{Introduction}
Starting with the work of Tutte \cite{Tutte_1963} in the sixties, the counting of planar maps has been an active topic in combinatorics. 
An important tool when studying maps are bijections between certain families of maps and certain families of (decorated) trees, first formulated by \cite{cori1981planar,schaeffer1998conjugaison}, later extended to more general cases (e.g. \cite{BDFG2004}).
These bijective approaches often give clearer insight, not only explaining why some generating functions are simple, but potentially also useful to study properties of the maps and even scaling limits (see for example \cite{Miermont2013BrownianSphere}).
These bijections have been extended or specialized to deal with the topology of the disk instead of the sphere (e.g. \cite{Bouttier_2009_quad_with_boundary}) or when extra decorations are added to the map (e.g. \cite{LeGall2011Scaling}).
\\\\
In physics, Ferrari recently motivated that the study of random (possibly) self-overlapping surfaces with constant curvature is essential to study so-called JT-gravity in the UV-limit \cite{Ferrari_Random_Disk_overview_2025,Ferrari_Random_Disk_lattice_2024}. 
Although for the physics community the hyperbolic case is somewhat more relevant, for simplicity, we will focus on the on flat case with topology of the disk.\footnote{For hyperbolic surfaces, there are also tree bijections under development \cite{BZ_hyperbolic_tree_2025}, but the toolbox is not yet large enough to study non-geodesic boundaries in this setting.} 

In a very recent paper \cite{Budd_rectilinear_2025}, Budd proposes a regularization of exactly these flat self-overlapping disks with arbitrary boundary, by restricting the boundary of the disk to a fixed number of straight sides, only joined by corners of $\pi/2$ or $3\pi/2$. 
These disks are called rectilinear disks.

The moduli spaces of rectilinear disks are continuous,\footnote{This follows naturally from the fact that the side-lengths are unconstrained.} yet they can be naturally built from the discrete rigid quadrangulations. 
These rigid quadrangulations, which are the objects that we study in this paper, can be seen as the combinatorial building blocks of rectilinear disks and thus a regularization of a `uniform random flat disk'.

In his paper, Budd shows that the rigid quadrangulations are in bijection with colorful quadrangulations.
These colorful quadrangulations and their dual Eulerian orientations were first described in \cite{Bonichon_number_2017} and \cite{ElveyPrice2018counting}, but have recently been extensively studied by Bousquet-Mélou and Elvey Price \cite{Bousquet-Melou_Eulerian_2020,Bousquet-Melou_generating_2020,Bousquetmelou2025refinedenumerationplanareulerian}.
Notably they were able to derive the generating function of colorful quadrangulations (and some variants of them).
They derived these generating functions by algebraically solving a system of equations, yet the expressions they got motivate the search for a more direct, bijective proof.
In particular, the expressions contain Catalan numbers, which are famously involved in counting trees. 
This hints that either colorful quadrangulations or rigid quadrangulations (or both) could have a natural bijection to (decorated) trees.

In \cite{Bousquet-Melou_generating_2020}, the authors suggest a family of decorated trees that are equinumerous to an important building block for the counting of colorful quadrangulations,\footnote{In \cite{Bousquetmelou2025refinedenumerationplanareulerian} the authors even show that a certain statistic has a natural counterpart in this family of trees.} yet the question of finding a direct bijection is still open. 

In this paper, we construct a bijection between rigid quadrangulations and some specific decorated trees, which differ from the suggested trees in \cite{Bousquet-Melou_generating_2020} and \cite{Bousquetmelou2025refinedenumerationplanareulerian}.\footnote{We will discuss possible connections in section~\ref{sec:sampling_outlook}.}
This allows us to provide a bijective proof of the generating series. 
This tree bijection might also open up new possibilities to study geometric properties of rigid quadrangulations, like the area or diameter, or even properties of colorful quadrangulations.

\subsection{Rigid quadrangulations}\FloatBarrier

\begin{figure}
    \centering
    \includegraphics[width=\linewidth]{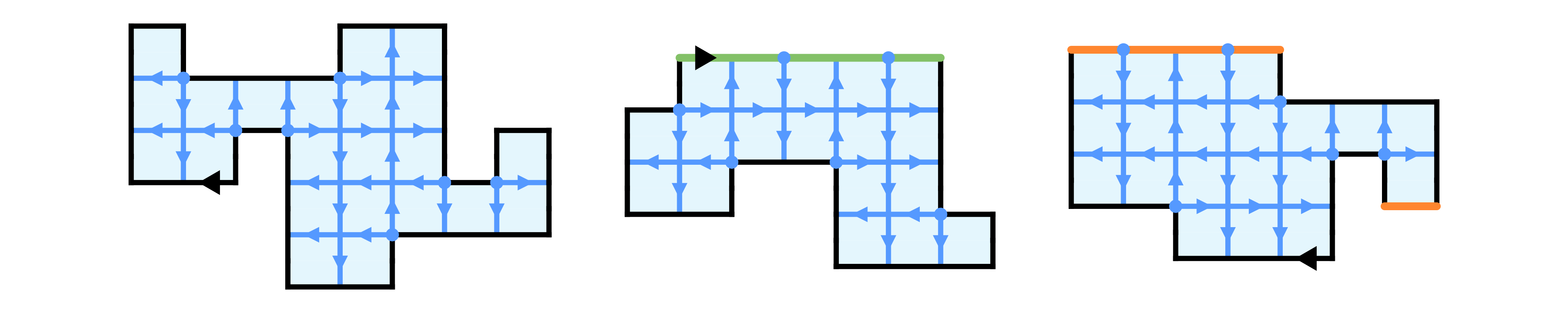}
    \caption{Three examples of rigid quadrangulations. Open sides will be colored orange, except when it is an open base, which we will color green.
    \\
    On the left: A base-$2$ rigid quadrangulation with 32 boundary edges split into 18 sides.  
    It has 11 convex corners and 7 concave corners.
    It has no open sides (orange), so it is complete.
    \\
    In the middle: A base-$(-3)$ rigid quadrangulation with 24 boundary edges split into 12 sides. 
    It has 8 convex corners and 4 concave corners.
    Although it has an open side (green), this is the base (side with the root), so it is still considered complete, but it gets a negative base-length.
    Note that for determining the base-length, we only consider downward rays (for open bases, these are the open rays that start on the base).
    \\
    On the right: A base-$3$ partial rigid quadrangulation with 24 boundary edges split into 12 sides.
    It has 8 convex corners and 4 concave corners.
    It has two open sides (orange), which are not the base, so it is considered partial. 
    Note that the open sides are horizontal and end in both ends in a convex corner, as required.
    They can have open rays, but this is not required.
    \label{fig:Examples}}
\end{figure}

Let's recall the terminology about rigid quadrangulations from \cite{Budd_rectilinear_2025} and add some new terminology.
A \emph{quadrangulation of the disk} is a rooted planar map, where the root is a distinguished oriented edge, such that all faces have degree $4$ except possibly the root face, which is the face on the left of the root.
We require the root face to be simple, such that all of its corners are incident to different vertices.
The contour of the root face is called the \emph{boundary} and the corresponding vertices and edges are called \emph{boundary vertices} and \emph{boundary edges}, while all other vertices and edges are \emph{inner vertices} and \emph{inner edges}.

A quadrangulation of the disk is \emph{flat} if each inner vertex has degree $4$.
Although this can be generalized in different models, we will only consider boundary vertices of degree at most $4$.
Boundary vertices of degree $2$ and $4$ are called convex and concave corners respectively, while a boundary vertex of degree $3$ is called straight.
\\\\
The flat condition, together with the bound on the degrees of boundary vertices, ensures that \emph{locally} we can embed a flat quadrangulation in the regular square grid. 
In general, this embedding does not exist globally: A flat quadrangulation can be immersed in the square grid, but parts of this immersion can overlap. 

In our figures, we will always immerse the disks into the square grid. 
For the clarity in the figures, we avoid examples with self-overlaps.\footnote{Occasionally we let two boundary vertices overlap. Note that these should be seen as two different vertices, even though they are drawn in the same place.}
Yet, we stress that all methods in this paper perfectly work for flat quadrangulations where the immersion has self-overlaps. 
In fact, the methods are fully independent of the immersion into the grid.\footnote{Although many names (horizontal, upward, etc.) will be inspired by the immersion.}
\\\\
A group of consecutive boundary edges between two corners, with only straight vertices in between is called a \emph{side}. 
The side that includes the root is called the \emph{base}.
Note that in \cite{Budd_rectilinear_2025} the definition of base is a bit more general, where it can consist of multiple sides.

Sides are \emph{horizontal} when there is an even number of corners between the root and the side.
If a horizontal side is between two convex corners, it can be marked as \emph{open}.
If a disk has open sides, you should think of it as being incomplete.
It can be completed by gluing disks to these open sides.
\\\\
To go from flat quadrangulations to rigid quadrangulations, we have to look at rays.
A \emph{ray} is a path of inner edges visiting only inner edges, except for its two endpoints, which are boundary vertices. At the inner vertices rays `go straight', meaning that at each inner vertex, two rays will cross.
Note that this implies that every inner edge belongs to a unique ray.

Finally, a \emph{rigid quadrangulation} is a flat quadrangulation such that
\begin{enumerate}
    \item The degree of boundary vertices is at most $4$.
    \item The root edge starts at a convex corner.
    \item Each ray is one of the following types: 
    \begin{itemize}
        \item a \emph{closed} ray, which connects a concave corner to a straight boundary vertex and is naturally oriented from the concave-corner towards the straight boundary vertex, or
        \item an \emph{open} ray, which connects two straight boundary vertices, together with an orientation, such that the ray starts on an open side.
    \end{itemize}
\end{enumerate}
The \emph{size} of an open side is the number of rays starting at that side plus one.

The rigid quadrangulation is said to be \emph{complete} when there is no open side, except for possibly the base. 
If there are non-base open sides, it is called \emph{partial}.
Note that a `rigid quadrangulation' in \cite{Budd_rectilinear_2025} has no open sides (including the base) and thus also no open rays. 
See Figure~\ref{fig:Examples} for some examples. 

\subsection{Based rigid quadrangulations}
A \emph{base-$p$ rigid quadrangulation} is a rigid quadrangulation, such the base starts and ends in a convex corner and $p\in\Z$ is the base-length defined below.

If the base is not an open side, we assign a positive base-length of $p$, where $p$ is the number of edges in the base.

If the base is an open side, we say it has a base-length $-(n+1)$, where $n$ is the number of rays starting on the base. Equivalently the length of an open base is minus its size.

Note the definition of base-length is quite different for positive and negatives base-lengths, yet we will see that they naturally extend each other.

Our definition of a base-$p$ rigid quadrangulation corresponds with the definition in \cite{Budd_rectilinear_2025} for $p>0$. 
In \cite{Budd_rectilinear_2025} this definition is extended to allow for $k$-fold bases, which we won't consider.
On the other hand, our extension to $p<0$ is not considered in \cite{Budd_rectilinear_2025}.
\\\\
In the figures, we will draw the base as a bottom side if it is not open and as a top side if it is open.
In accordance with this convention, we call a horizontal side a \emph{bottom side} when either
\begin{itemize}[nosep]
    \item the base is not open and the difference between the number of convex and concave corners between the base and the side is not a multiple of 4, or
    \item the base is open and the difference between the number of convex and concave corners between the base and the side is a multiple of 4.
\end{itemize}
Naturally, a horizontal side, which is not a bottom side, is a \emph{top side}.   
Similarly, a ray that ends on a bottom side, will be called \emph{downward}, while rays ending on a top side are \emph{upward}. 

In retrospect, we can now see why the definition of base-lengths for open and non-open base are each other's extensions: They count the length of the base, only considering downward rays, ignoring upward rays, where we add a minus sign for open bases. 
\\\\
The set of complete base-$p$ rigid quadrangulations with $2n$ corners\footnote{Equivalently, it has $n$ convex corners not incident to the base.} is denoted $\mathcal{R}^{(p)}_n$ and the corresponding generating functions are
\begin{equation}
    F^{(p)}(t) = \sum_{n=0}^\infty  |\mathcal{R}^{(p)}_{n}|\,t^n.
\end{equation}
By convention, we add a single base-$0$ rigid quadrangulation with $n=1$ and thus set $F^{(0)}(t) = t$. 
It is useful to see this `disk' as just a single convex corner.

The main theorem of this paper is as follows:
\begin{theorem}\label{thm:gen_fun_rigid}
    Let $R(t)=t-2t^2-4t^3-20t^4-\ldots\in\Z[[t]]$ be the unique formal power series solution to
    \begin{align}\label{eq:R_def}
        \sum_{n\geq0}\frac{1}{n+1}\binom{2n}{n}\binom{2n}{n} R(t)^{n+1} = t,
    \end{align}
    with $[t^0]R(t)=0$.

    The generating function of the number of (complete) base-$p$ rigid quadrangulations is given by
    \begin{align}
        F^{(p)}(t) =
        \begin{dcases}
            \sum_{n\geq p}\frac{1}{n+1}\binom{2n}{n}\binom{2n-p}{n-p} R(t)^{n+1}&\textrm{ for }p>0\\
            \sum_{n\geq 1}\frac{1}{n+1}\binom{2n}{n}\binom{2n-p-1}{n-p} R(t)^{n+1}&\textrm{ for }p<0\\
        \end{dcases}.
    \end{align}
\end{theorem}
We will prove this theorem using a tree bijection.

It is already noted in \cite{Budd_rectilinear_2025} that $F^{(1)}(t)$ is related to the generating function of rigid quadrangulations that are not based (just rooted at a convex corner, with no other restrictions):
\begin{align}
    Z(t)=\frac{F^{(1)}(t)-t^2}{2t}.
\end{align} 
Furthermore, it follows directly from the definition of a negative base-length that $F^{(-1)}(t)=\sum_{p=1}^{\infty}F^{(p)}(t)$, which is a good consistency check for Theorem~\ref{thm:gen_fun_rigid}.

\subsection{Double based rigid quadrangulations}\FloatBarrier

We can use the tree bijection to expand on this result to some restricted rigid quadrangulations.

\begin{figure}
    \centering
    \includegraphics[width=\linewidth]{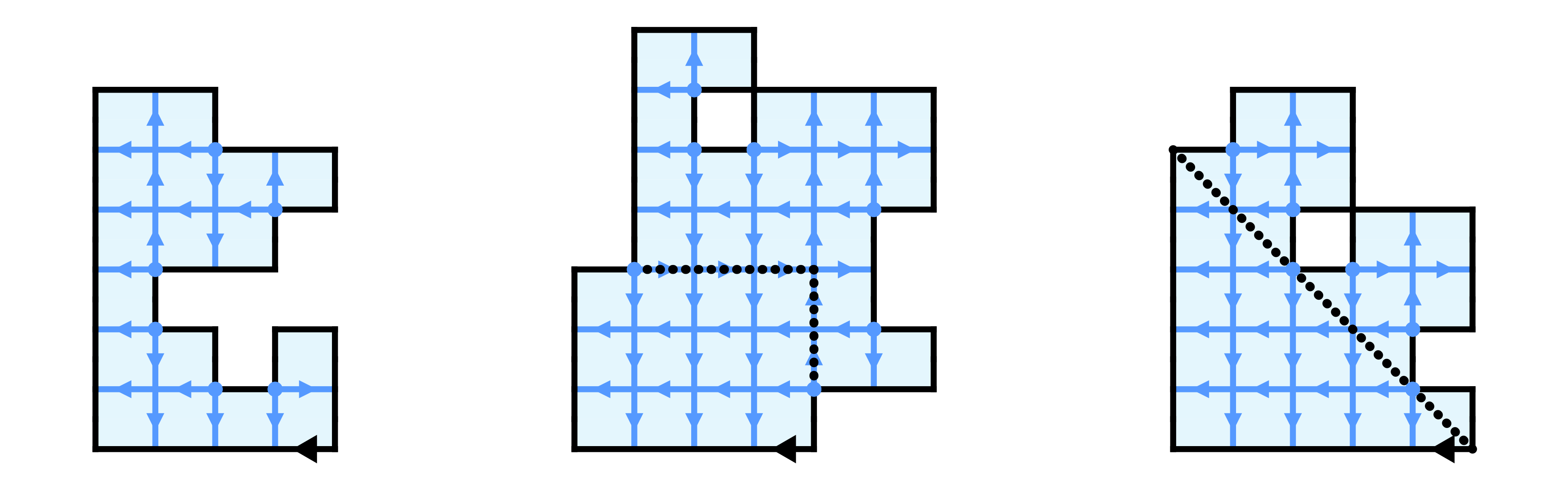}
    \caption{A B-type, a C-type and a $\Delta$-type rigid quadrangulation.
    The B-type rigid quadrangulation has base-length $4$ and co-base-length $6$.
    The C-type rigid quadrangulation has base-length $4$ and co-base-length $3$.
    The $\Delta$-type rigid quadrangulation has base-length $5$ and co-base-length $5$ and degeneracy $3$.
    \label{fig:BCD}}
\end{figure}

A \emph{double based rigid quadrangulation} (or a \emph{B-type rigid quadrangulation}, due to its correspondence to B-patches in \cite{Bousquet-Melou_generating_2020}) is a rigid quadrangulation, such that base and the side following the base clockwise around the boundary begin and end in convex corners. 
The side following the base in clockwise direction will be called the \emph{co-base}. 
We will only consider double based rigid quadrangulations where the base and co-base are closed sides, and we will only assign positive (co-)base-lengths to them, equal to the number of edges that they contain.

A \emph{C-type rigid quadrangulation} is a B-type rigid quadrangulation, such that the base and co-base are two sides of a rectangle embedded in the disk.
More precisely, in a C-type rigid quadrangulation with base-length $p$ and co-base-length $q$, the $(p-1)$ downward rays ending at the base consist of at least $q$ edges, or equivalently, the $(q-1)$ horizontal rays ending at the co-base consist of at least $p$ edges.

A \emph{$\Delta$-type rigid quadrangulation} is a B-type rigid quadrangulation, such that the base and co-base are two sides of a right-angled triangle embedded in the disk.
More precisely, in a $\Delta$-type rigid quadrangulation with base-length $p$ and co-base-length $q$, the $k$th downward ray  ending at the base ($1\leq k\leq p-1$, where the ray closest to the root is the first, following in clockwise order) consists of at least $k q/p$ edges. 
Or equivalently, the $k$th horizontal ray ending at the co-base ($1\leq k\leq q-1$, where the ray closest to the base is the first, following in clockwise order) consists of at least $k p/q$ edges. 
The \emph{degeneracy} of a $\Delta$-type rigid quadrangulation is the number of concave corners that lay on the hypotenuse of the triangle, plus one.
Or again more precisely, it is the number of values of $k$ for which the $k$th downward ray ending at the base exactly has $k q/p$ edges, plus one, or equivalently, it is the number of values of $k$ for which the $k$th horizontal ray ending at the co-base exactly has $k p/q$ edges, plus one. 
Note that when $p$ and $q$ are co-prime, the degeneracy is guaranteed to be one.

\begin{theorem}\label{thm:D_B_C_gfs}
    Let $\Delta(t,x,y)\in\Q[x,y][[t]]$ be the generating function of (complete) $\Delta$-type rigid quadrangulations weighted by the inverse of their degeneracies, where $t$ counts the number of convex corners, not incident to the base or the co-base, $x$ counts the co-base-length and $y$ counts the base-length.
    Then
    \begin{align}
        \Delta(t,x,y)=\sum_{n\geq0}\sum_{j=0}^n\sum_{i=0}^n\frac{1}{n+1}\binom{2n-i}{n}\binom{2n-j}{n}x^{i+1}y^{j+1}R^{n+1},
    \end{align} 
    where $R=R(t)$ is the solution of equation \eqref{eq:R_def} as before.

    Moreover, let $B(t,x,y)\in\Z[x,y][[t]]$ and $C(t,x,y)\in\Z[x,y][[t]]$ be the generating functions of (complete) B- and C-type rigid quadrangulations respectively, where $t$ counts the number of convex corners, not incident to the base or the co-base, $x$ counts the co-base-length and $y$ counts the base-length.
    Then
    \begin{align}
        B(t,x,y)&=\exp(\Delta(t,x,y))-1,\\
        C(t,x,y)&=1-\exp(-\Delta(t,x,y)).
    \end{align}
\end{theorem}

The first few terms of the series above are 
\begin{align}
    \Delta(t,x,y)&=xyt+(xy^2+x^2(y+y^2/2))t^2+(x(2y^2+2y^3)+x^2(2y+y^2+y^3)+x^3(2y+y^2+y^3/3))t^3+\ldots\\
    B(t,x,y)&=xyt+(xy^2+x^2(y+y^2))t^2+(x(2y^2+2y^3)+x^2(2y+y^2+2y^3)+x^3(2y+2y^2+y^3))t^3+\ldots\\
    C(t,x,y)&=xyt+(xy^2+x^2y)t^2+(x(2y^2+2y^3)+x^2(2y+y^2)+2x^3y)t^3+\ldots \quad.
\end{align}

\subsection{Sketch of proof and outline}\FloatBarrier

\begin{figure}
    \centering
    \includegraphics[width=\linewidth]{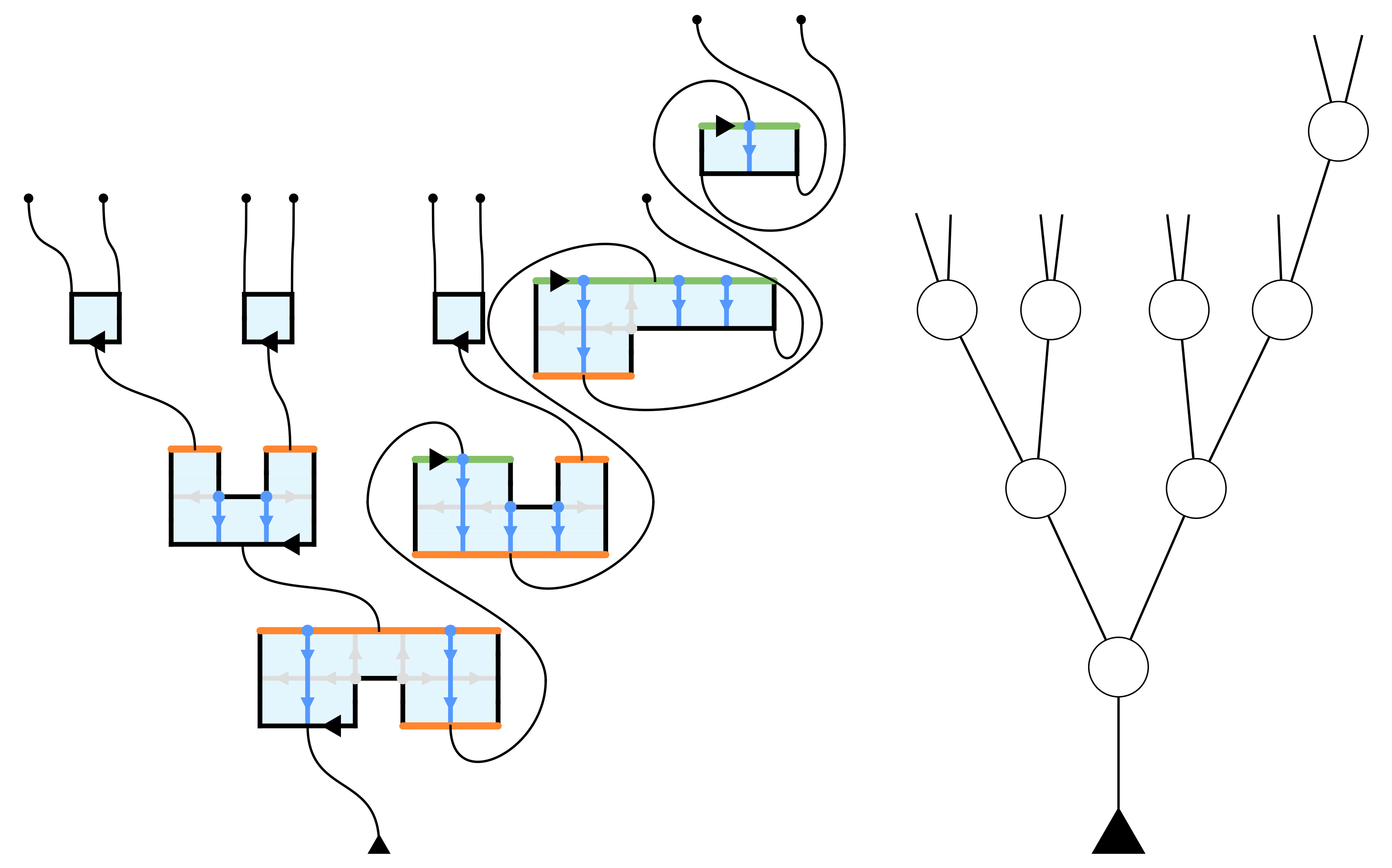}
    \caption{Idea of the tree bijection. 
    On the left we have the decomposition of the left disk from Figure~\ref{fig:Examples} into slices. 
    On the right we have the tree structure of this decomposition.  
    \label{fig:Tree_structure}}
\end{figure}

As shown in Figure~\ref{fig:Tree_structure}, we will obtain from rigid quadrangulation a tree structure, by decomposing the disk in slices, similar to the row-by-row exploration in \cite{Budd_rectilinear_2025}, where the convex corners\footnote{Except those adjacent to the root.} will act as leaves.
Since the decomposition of different disks can give the same tree structure, this tree itself is not enough for a bijection. 
Later we will see what data should be added to the tree to make this a bijection. 

To do this, we first need to define how exactly we glue / cut rigid quadrangulations and how we can describe the slices that we get after the decomposition.
We will do this in section~\ref{sec:gluing}.

In section~\ref{sec:H_trees}, we will define H-trees, which are the binary trees from Figure~\ref{fig:Tree_structure}, decorated with the correct data. 
In this section we will prove that these H-trees are in bijection with rigid quadrangulations.

The H-trees are not easy to analyze directly. 
It appears that we can do a small manipulation of the added data, such that the result is somewhat easier.
These manipulated trees are called Q-trees, and they are the topic of section~\ref{sec:Q_trees}.

In section~\ref{sec:enumeration}, we will actually find the generating function of Q-trees, from which the generating function of H-trees (and thus the generating function of rigid quadrangulations) immediately follows, proving Theorem~\ref{thm:gen_fun_rigid}.

In section~\ref{sec:BCD}, we apply the methods of previous sections to the B-, C- and $\Delta$-type rigid quadrangulations, proving Theorem~\ref{thm:D_B_C_gfs}. 

Finally, in section~\ref{sec:sampling_outlook}, we briefly touch upon sampling uniform random rigid quadrangulations using the tree bijection and give some outlook on open questions.

\FloatBarrier
\subsection*{Acknowledgments}
This work is supported by the VIDI programme with project number VI.Vidi.193.048, which is financed by the Dutch Research Council (NWO).
Furthermore, the author wants to thank Timothy Budd and Andrew Elvey Price for the fruitful discussions on this topic.

\FloatBarrier
\clearpage
\section{Gluing partial rigid quadrangulations}\label{sec:gluing}
\FloatBarrier

In \cite{Budd_rectilinear_2025} it is described how partial rigid quadrangulations can be glued.
We will use that gluing prescription (which we will repeat below for completeness) only for gluing disks to open top sides, while we use a slightly different gluing rule for gluing disks to open bottom sides.

\begin{figure}[p]
    \centering
    \includegraphics[width=\linewidth]{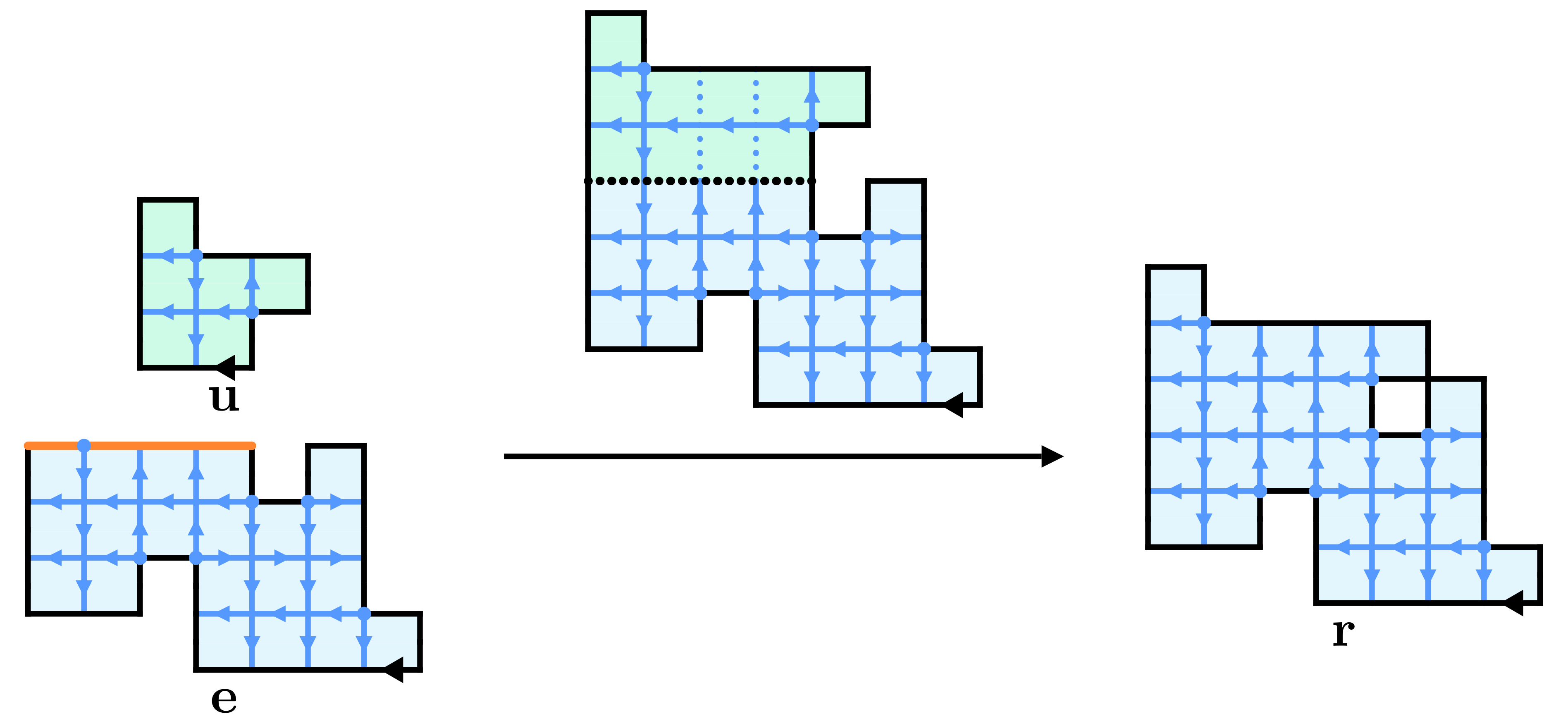}
    \caption{The gluing of the disk $\mathbf{u}$ to an open top side of $\mathbf{e}$.
    Note that only the number of downward rays should match on the two sides that are glued. 
    The upward rays of $\mathbf{e}$ are extended into $\mathbf{u}$.
    \label{fig:Positive_gluing}}
\end{figure}

To an open \emph{top} side $s$ of size $\ell$, we can naturally glue a base-$\ell$ rigid quadrangulation $\mathbf{u}$ (see Figure~\ref{fig:Positive_gluing}).
We do this by identifying the base of $\mathbf{u}$ with the side $s$ such that the $(\ell-1)$ downward rays in $\mathbf{e}$ that start at $s$ align with the $(\ell-1)$ downward rays of $\mathbf{u}$ ending at its base.
Consecutively, we extend the upward rays ending in $s$ vertically, until they hit a horizontal boundary of $\mathbf{u}$.
Finally, we merge the quadrangles above and below $s$, by removing the edges and vertices of $s$.

\begin{figure}[p]
    \centering
    \includegraphics[width=\linewidth]{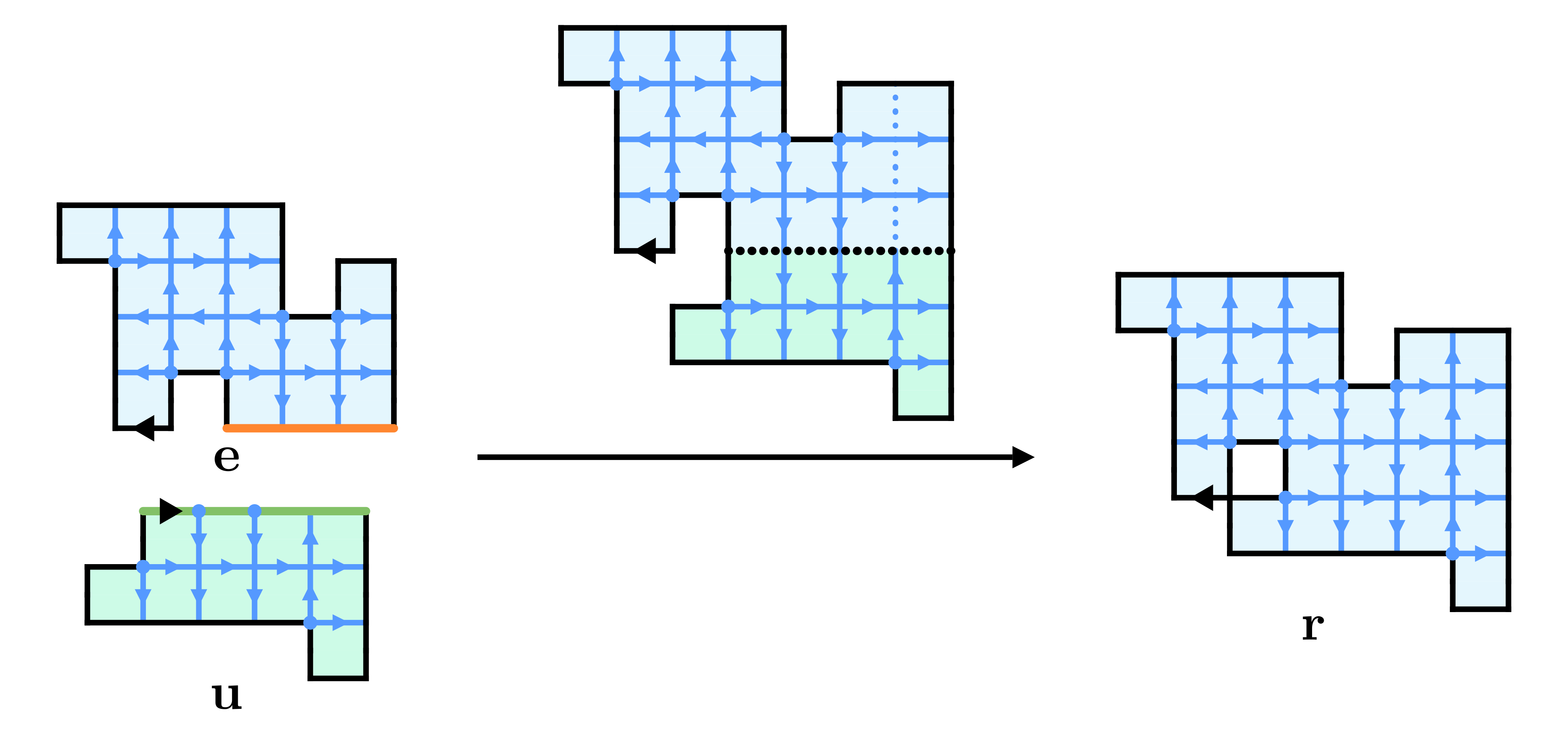}
    \caption{The gluing of the disk $\mathbf{u}$ to an open bottom side of $\mathbf{e}$.
    Note that again only the number of downward rays should match on the two sides that are glued. 
    The upward rays of $\mathbf{u}$ are extended into $\mathbf{e}$.
    \label{fig:Negative_gluing}}
\end{figure}

To an open \emph{bottom} side $s$ of size 1 (so no upward rays) and length $\ell$, we can naturally glue a base-$(-\ell)$ rigid quadrangulation $\mathbf{u}$.
We do this by identifying the base of $\mathbf{u}$ with the side $s$ such that the $(-\ell-1)$ downward rays in $\mathbf{e}$ that end at $s$ align with the $(-\ell-1)$ downward rays of $\mathbf{u}$ starting at its base.
Consecutively, we extend the upward rays ending at the base of $\mathbf{u}$ vertically, until they hit a horizontal boundary of $\mathbf{e}$.
Finally, we merge the quadrangles above and below $s$, by removing the edges and vertices of $s$.
See Figure~\ref{fig:Negative_gluing} for an example.

Note that we have no gluing rule for open bottom sides with sizes larger than 1, as they will not appear in our decomposition.

We call a base-$p$ rigid quadrangulation $\mathbf{e}$ a submap of base-$p$ rigid quadrangulation $\mathbf{r}$ (denoted $\mathbf{e}\subset\mathbf{r}$), if there exists a tuple of rigid quadrangulations $(\mathbf{u}_s)_{s\in S}$, where $S$ is a subset of the non-base open sides of $\mathbf{e}$, such that gluing all $\mathbf{u}_s$ to the corresponding $s$ gives $\mathbf{r}$.

Note that due to the different gluing rule for bottom open sides, the notion of submap in this paper is different from the definition in \cite{Budd_rectilinear_2025}.

Given $\mathbf{e}\subset\mathbf{r}$, the rigid quadrangulations $\mathbf{u}_s$ are uniquely determined, so, for fixed $\mathbf{e}$, the set $\{ \mathbf{r}\in\mathcal{R}_n^{(p)} : n>0 , \mathbf{e}\subset \mathbf{r}\}$ is in bijection with tuples $(\mathbf{u}_s)_s$ of rigid quadrangulations of appropriate base-lengths.

We can now modify \cite[Lemma~11]{Budd_rectilinear_2025} to the current gluing rules:
\begin{lemma}\label{lem:minimal_signed_submaps}
    Every complete base-$p$ rigid quadrangulation $\mathbf{r}$ with $p>0$ contains exactly one of the positive minimal submaps from Figure~\ref{fig:Minimal_subdisks_pos} and described below:

    \begin{figure}[p]
        \centering
        \includegraphics[width=\linewidth]{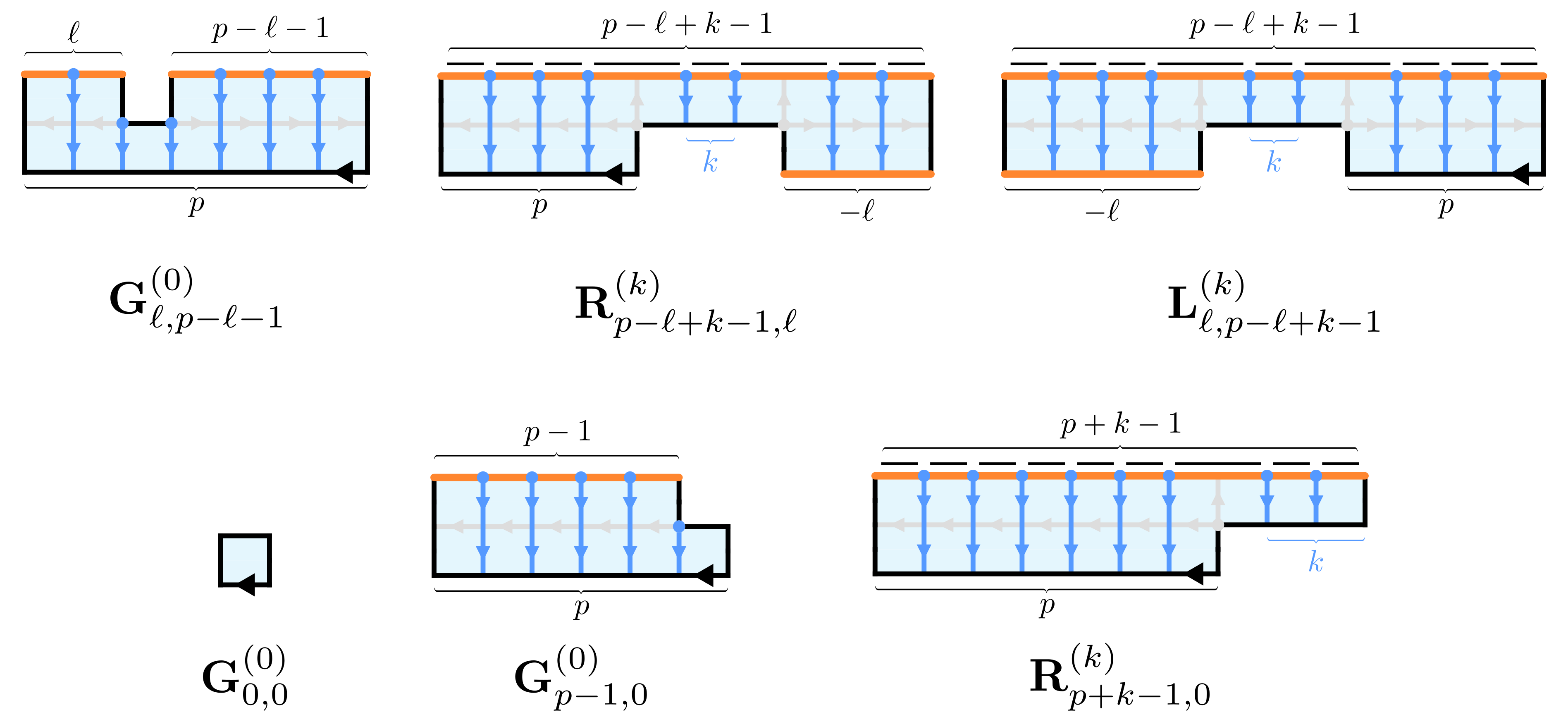}
        \caption{Top row: The three different types of positive minimal submaps. 
        Bottom row: Some degenerate cases where we get one or two convex corners instead of an open side.
        Note that we ignore upward rays when determining the ``lengths'' of sides.
        Also note that for $k$, we count the number of downward rays hitting the horizontal side in the middle (so $k=2$ in the two cases depicted in the top row), where we also add one in the degenerate case (so $k=3$ in the last case of the bottom row).
        \label{fig:Minimal_subdisks_pos}}
    \end{figure}

    \begin{itemize}
        \item $\mathbf{G}^{(0)}_{\ell,p-\ell-1}$, which is u-shaped and has two open top sides of sizes $\ell$ and $p-\ell-1$ for some $\ell = 0,\ldots,p-1$. 
        In the boundary cases $\ell=0$ or $\ell=p-1$ (or both) there is a convex corner instead of an open top side.
        \item $\mathbf{R}^{(k)}_{p-\ell+k-1,\ell}$, where $\ell\leq0$ and $k\geq0$ and $(k,\ell)\neq(0,0)$. 
        It is n-shaped, has an open top side of size $p-\ell+k-1$ and an open bottom side of size 1 and length $-\ell\geq 1$ to the right of the base. 
        In the boundary case $\ell=0$, there is a convex corner instead of an open bottom side.  
        \item $\mathbf{L}^{(k)}_{\ell,p-\ell+k-1}$ is the same as $\mathbf{R}^{(k)}_{p-\ell+k-1,\ell}$, except that the open bottom side (or convex corner) is to the left of the base. 
    \end{itemize}
    Every complete base-$p$ rigid quadrangulation $\mathbf{r}$ with $p<0$ contains exactly one of the following negative minimal submaps from Figure~\ref{fig:Minimal_subdisks_neg} and described below:

    \begin{figure}[p]
        \centering
        \includegraphics[width=\linewidth]{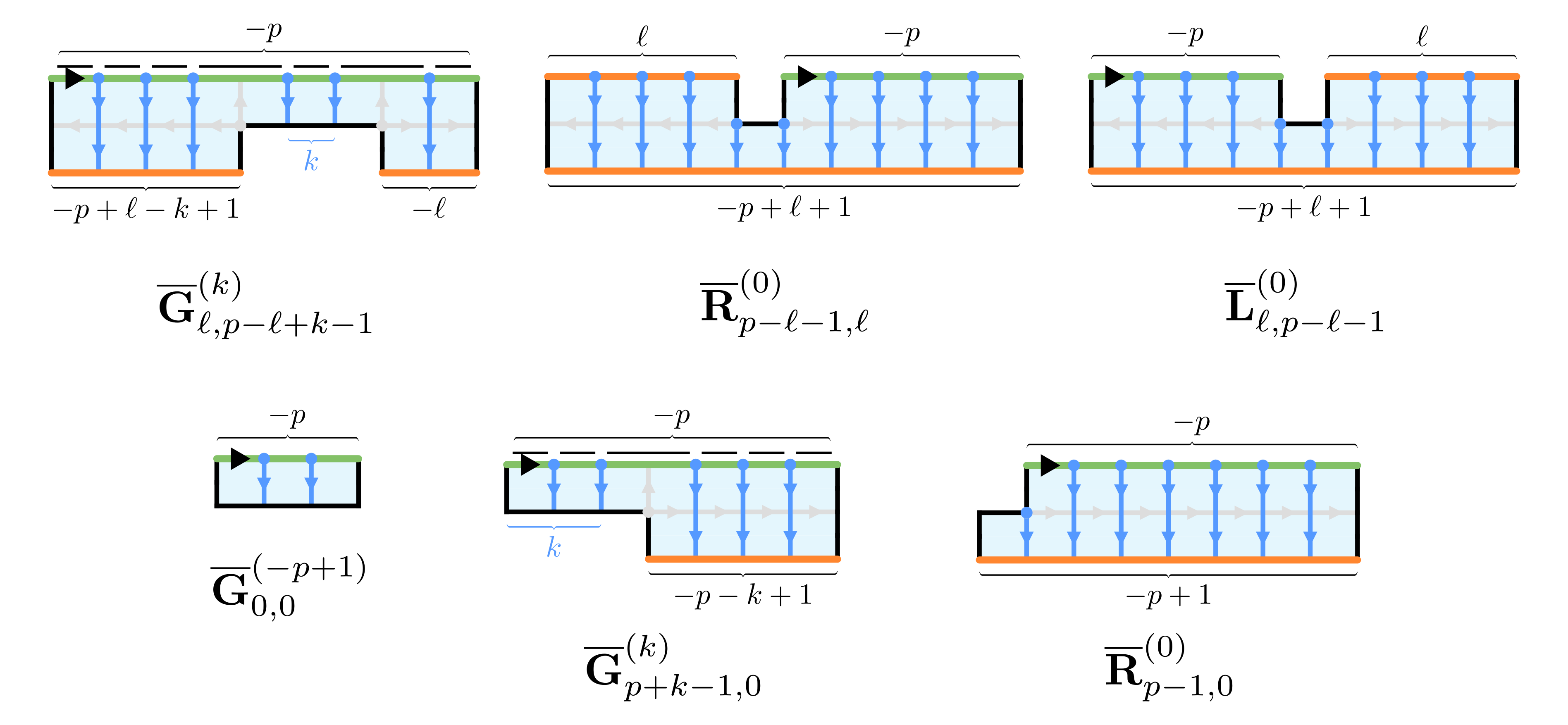}
        \caption{Top row: The three different types of negative minimal submaps. 
        Bottom row: Some degenerate cases where we get one or two convex corners instead of an open side.
        Note again that we ignore upward rays when determining the ``lengths'' of sides.
        Again, note that for $k$, we count the number of downward rays hitting the horizontal side in the middle (so $k=2$ in the upper left case), where we also add one per degeneracy (so $k=4$ and $k=3$ respectively in the first two cases of the bottom row).
        \label{fig:Minimal_subdisks_neg}}
    \end{figure}

    \begin{itemize}
        \item $\overline{\mathbf{G}}^{(k)}_{\ell,p-\ell+k-1}$ which is n-shaped and has two open bottom sides of size 1 and lengths $p-\ell+k-1$ and $\ell$ for some $\ell = p+k-1,\ldots,0$ and $k=0,\ldots,-p+1$. 
        In the boundary cases $\ell=0$ or $\ell=p+k-1$ (or both) there is a convex corner instead of an open bottom side and $k$ will be at least $1$ (resp. $2$).
        \item $\overline{\mathbf{R}}^{(0)}_{p-\ell-1,\ell}$ where $\ell\geq0$. 
        It is u-shaped, has an open bottom side of size 1 and length $p-\ell-1$ and an open top side of size $\ell\geq 1$ to the left of the base.\footnote{This is on the right as seen downwards from the base, explaining the convention.} 
        In the boundary case $\ell=0$, there is a convex corner instead of an open top side. 
        \item $\overline{\mathbf{L}}^{(0)}_{\ell,p-\ell-1}$ is the same as $\overline{\mathbf{R}}_{p-\ell-1,\ell}$, except that the open top side appears to the right of the base.
    \end{itemize}
\end{lemma}

\begin{proof}
    The proof is a straightforward adaptation from the proof of \cite[Lemma~11]{Budd_rectilinear_2025}.
\\\\
    Let's start with the case $p>0$. 
    We consider a base-$p$ rigid quadrangulation $\mathbf{r}$. 
    Let $v_{\mathrm{L}}$ and $v_{\mathrm{R}}$ be the first boundary vertices adjacent to the base on the left and on the right.
    One of the following mutually exclusive cases must apply:
    \begin{itemize}[nosep]
        \item $v_{\mathrm{R}}$ and $v_{\mathrm{L}}$ are convex corners, in which case we have $\mathbf{r} = \mathbf{G}^{(0)}_{0,0}$, and we are done.
        \item Exactly one of $\{v_{\mathrm{L}},v_{\mathrm{R}}\}$ is a convex corner, which means that the other is straight. 
        The convex corner must be followed by a concave corner, giving $\mathbf{G}^{(0)}_{p-1,0} \subset \mathbf{r}$ or $\mathbf{G}^{(0)}_{0,p-1} \subset \mathbf{r}$.
        \item $v_{\mathrm{R}}$ and $v_{\mathrm{L}}$ are straight vertices. 
        In this case, they are both endpoints of horizontal rays, which start in adjacent concave corners, giving $\mathbf{G}^{(0)}_{\ell,p-\ell-1}\subset \mathbf{r}$ for some $\ell=1,\ldots,p-2$.
        \item $v_{\mathrm{R}}$ is a concave corner, which means that $v_{\mathrm{L}}$ is straight.
        If we continue following the boundary from $v_{\mathrm{R}}$ to the right, the next corner occurs after $k\geq0$ straight vertices. 
        It can be
        \begin{itemize}
            \item a convex corner, in which case we have $\mathbf{R}^{(k+1)}_{p+k,0} \subset \mathbf{r}$;
            \item a concave corner. 
            The horizontal ray that starts in this concave corner will cross $(-\ell-1)\geq0$ downward rays and an arbitrary number of upward rays in arbitrary order before ending at a straight boundary vertex.
            This gives $\mathbf{R}^{(k)}_{p-\ell+k-1,\ell} \subset \mathbf{r}$ (the upward rays will reappear after the gluing).
        \end{itemize}
        \item $v_{\mathrm{L}}$ is a concave corner, which means that $v_{\mathrm{R}}$ is straight.
        We can mirror the steps from the case above, giving $\mathbf{L}^{(k)}_{p-\ell+k-1,\ell} \subset \mathbf{r}$ for some $k\geq0$ and $\ell\leq0$, except $k=\ell=0$. 
    \end{itemize}

    \noindent We can do a similar analysis when $p<0$.
    Note that $v_{\mathrm{L}}$ here is on the left as seen from the root edge, so in the drawings it will be on the right.
    One of the following mutually exclusive cases must apply:
    \begin{itemize}[nosep]
        \item $v_{\mathrm{R}}$ and $v_{\mathrm{L}}$ are convex corners, in which case we have $\mathbf{r} = \overline{\mathbf{G}}^{(-p+1)}_{0,0}$, and we are done.
        \item Exactly one of $\{v_{\mathrm{L}},v_{\mathrm{R}}\}$ is a convex corner, which means that the other is straight. 
        After $k-1\geq0$ straight boundary vertices, the convex corner will be followed by a concave corner, giving $\overline{\mathbf{G}}^{(k)}_{p+k+1,0} \subset \mathbf{r}$ or $\overline{\mathbf{G}}^{(k)}_{0,p+k+1} \subset \mathbf{r}$.
        \item $v_{\mathrm{R}}$ and $v_{\mathrm{L}}$ are straight vertices. 
        In this case, they are both endpoints of horizontal rays, which start in concave corners that have $k\geq0$ straight boundary vertices between them, giving $\overline{\mathbf{G}}^{(k)}_{\ell,p-\ell+k-1}\subset \mathbf{r}$ for some $\ell=p+k,\ldots,-1$.
        \item $v_{\mathrm{R}}$ is a concave corner, which means that $v_{\mathrm{L}}$ is straight.
        Then the horizontal ray starting at $v_{\mathrm{R}}$ will end at a straight $v_{\mathrm{L}}$. 
        If we continue following the boundary from $v_{\mathrm{R}}$, the next corner can be
        \begin{itemize}
            \item a convex corner, in which case we have $\overline{\mathbf{R}}^{(0)}_{p-1,0} \subset \mathbf{r}$;
            \item a concave corner. 
            The horizontal ray that starts in this concave corner will cross $\ell-1\geq0$ downward rays and an arbitrary number of upward rays in arbitrary order before ending at a straight boundary vertex. 
            This gives $\overline{\mathbf{R}}^{(0)}_{p-\ell-1,\ell} \subset \mathbf{r}$ (the upward rays will reappear after the gluing).
        \end{itemize}
        \item $v_{\mathrm{L}}$ is a concave corner, which means that $v_{\mathrm{R}}$ is straight. We can mirror the steps from the case above, giving $\overline{\mathbf{L}}^{(0)}_{\ell,p-\ell-1} \subset \mathbf{r}$ for some $\ell\leq0$. 
    \end{itemize}
\end{proof}

\noindent For each of these minimal submaps $\mathbf{Z}^{(k)}_{a,b}$, ($\mathbf{Z}\in\{\mathbf{G},\mathbf{R},\mathbf{L}\}$) the base-length is given by $p=a+b-k+1$, and we say that the submap has signature $(p,a,b,k)$.

\begin{remark}
    Note that $a$ and $b$ are the base-lengths of the rigid quadrangulations that are to be glued to the two open sides. 
    Also note that the type of minimal submap (the `value' of $\mathbf{Z}$) is uniquely defined by its signature (see Table~\ref{tab:signatures}).
    The set of allowed signatures is given by
    \begin{align}\label{eq:signatures}
        \left\{(p,a,b,k)\in\Z_{\neq0}\times\Z^2\times\Z_{\geq0}: p=a+b-k+1 \textrm{ and } \ind_{p<0}\geq\ind_{a\leq0}+\ind_{b\leq0} \implies k=0\right\}.
    \end{align}
    Note that $\ind_{p<0}\geq\ind_{a\leq0}+\ind_{b\leq0}$ is a complicated expression, which just means that the submap is always u-shaped or in other words, that the horizontal side in the middle will always be a top side.
\end{remark}

\begin{table}[h]
\caption{Table of signatures $(p,a,b,k)$ of minimal submaps, where $p=a+b-k+1$. 
Disallowed signatures are indicated by I when $p=a+b-k+1$ cannot be satisfied and by II when $\ind_{p<0}\geq\ind_{a\leq0}+\ind_{b\leq0}$ for $k>0$.
\\
\textsuperscript{*}Note that for $p<0,a=b=0$, we must have $k\geq2$.}
\label{tab:signatures}
\centering
\begin{tabular}{rr||cc|cc|cc}
    \multicolumn{2}{c||}{}&\multicolumn{2}{c|}{$a<0$}&\multicolumn{2}{c|}{$a=0$}&\multicolumn{2}{c}{$a>0$}\\
    \multicolumn{2}{c||}{}&$p<0$&$p>0$&$p<0$&$p>0$&$p<0$&$p>0$\\
    \hline\hline
    \multirow{2}*{$b<0$}&$k=0$ &$\overline{\mathbf{G}}^{(0)}_{a,b}$ &I&$\overline{\mathbf{L}}^{(0)}_{0,b}$      
    &I & $\overline{\mathbf{L}}^{(0)}_{a,b}$ &$\mathbf{R}^{(0)}_{a,b}$ \\   
    &$k>0$      &$\overline{\mathbf{G}}^{(k)}_{a,b}$ &I& $\overline{\mathbf{G}}^{(k)}_{0,b}$&I&II&$\mathbf{R}^{(k)}_{a,b}$ \\
    \hline
    \multirow{2}*{$b=0$}&$k=0$ &$\overline{\mathbf{R}}^{(0)}_{a,0}$ &I &I     
    &$\mathbf{G}^{(0)}_{0,0}$&I &$\mathbf{G}^{(0)}_{a,0}$ \\  
    &$k>0$      &$\overline{\mathbf{G}}^{(k)}_{a,0}$&I & $\overline{\mathbf{G}}^{(k)}_{0,0}$ \textsuperscript{*} &I& II  &$\mathbf{R}^{(k)}_{a,0}$ \\
    \hline
    \multirow{2}*{$b>0$}&$k=0$ &$\overline{\mathbf{R}}^{(0)}_{a,b}$  
    &$\mathbf{L}^{(0)}_{a,b}$&I& $\mathbf{G}^{(0)}_{0,b}$&I &$\mathbf{G}^{(0)}_{a,b}$ \\ 
    &$k>0$      &II     
    &$\mathbf{L}^{(k)}_{a,b}$&II &$\mathbf{L}^{(k)}_{0,b}$ &II&II \\
\end{tabular}    
\end{table}

The idea of the tree bijection is that in each step of decomposition using the minimal submaps, we get two parts, which are both either a convex corner, where the recursion stops, or a new rigid quadrangulation, which can be decomposed further.
This naturally gives a binary plane tree. 
The claim is that the signatures of the minimal submaps involved in this decomposition is exactly the decoration that the tree needs, to make a bijection.

\FloatBarrier
\section{H-trees}\label{sec:H_trees}
\FloatBarrier
We will now define the decorated tree that we need for the bijection, following the row-by-row decomposition sketched above. 
While the tree itself is just a simple rooted binary plane tree, the decoration will have to follow a more complicated set of rules, matching the complicated set of allowed signatures given in \eqref{eq:signatures}.
\\\\
Consider a rooted binary tree $T$ with root $v_0$ and denote the vertex and edge set of the tree by $V(T)$ and $E(T)$ respectively.
We call vertices of degree $>1$ \emph{internal} and denote the set of internal vertices by $V_I(T)\subset V(T)$.
The root induces an orientation in the tree such that for each internal vertex $v\in V_I(T)$ we can find a parent edge $v^-$ and a left child edge $v^\textrm{L}$ and a right child edge $v^\textrm{R}$, which are all elements of $E(T)$. 
We will also call $v^-$ the parent edge of $v^\textrm{L}$ and $v^\textrm{R}$ and similarly, $v^\textrm{L}$ and $v^\textrm{R}$ are also called the left/right child edges of $v^-$.

\begin{definition}
    We define a \emph{partition tree} of degree $n$ and base-length $p$ to be a tuple $(T,v_0,f_V,f_E)$. 
    Here, $T$ is a rooted binary plane tree with $n$ leaves and root $v_0$. 
    The functions $f_V: V_I(T)\rightarrow \Z_{\geq0}$ and $f_E: E(T)\rightarrow\Z$ are integer labelings on internal vertices and the edges respectively, such that
    \begin{enumerate}[label=(P.\arabic*)]
        \item\label{enumitem:sum_condition} for all $v\in V_I(T)$ 
        \begin{align}\label{eq:fV_fE_relation}
            f_V(v)=f_E(v^\textrm{L})+f_E(v^\textrm{R}) - f_E(v^-)+1.
        \end{align}
        Heuristically, the labels on the vertices indicate the `excess' of edge labels, offset by one.
        \item\label{enumitem:root_label} $f_E(e_0)=p$ where $e_0$ is the edge incident to the root $v_0$.
    \end{enumerate}
\end{definition}

We call an edge $e$ positive when $f_E(e)$ is positive.
We call an internal vertex $v$ a \emph{bottom vertex} when $\ind_{f_E(v^-)\leq0}<\ind_{f_E(v^\textrm{L})\leq0}+\ind_{f_E(v^\textrm{R})\leq0}$.
\footnote{Heuristically, we can say that at least one non-positive edge is newly created at a bottom vertex.}
Otherwise, the internal vertex is called a \emph{top vertex}.
\footnote{A top vertex will correspond in the bijection to a top side in the rigid quadrangulation, thus the naming. 
A bottom vertex will correspond to a bottom side, except when both the vertex and at least one of the child edges have label $0$, in which case it will unfortunately correspond to a top side.}

\begin{figure}
    \centering
    \includegraphics[width=\linewidth]{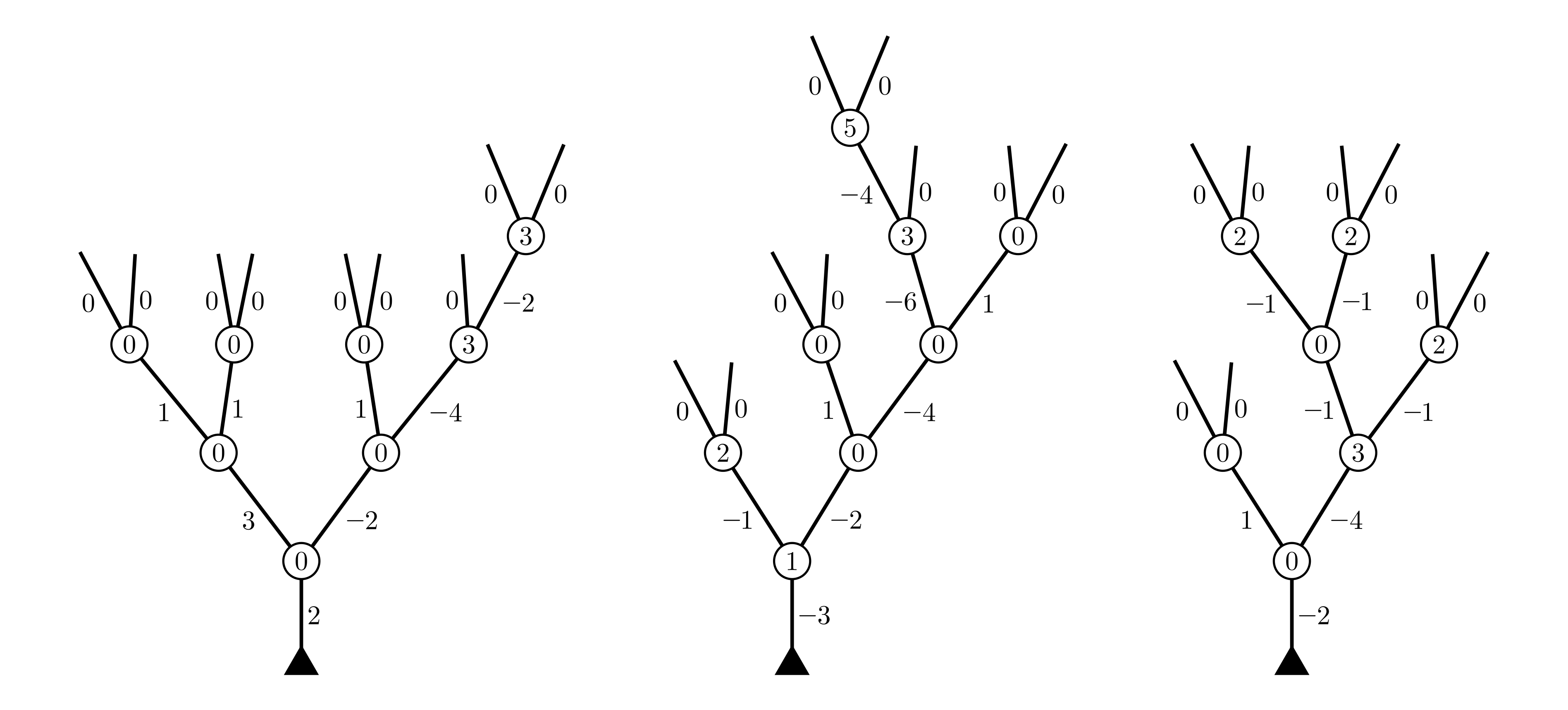}
    \caption{Some examples of H-trees, with base-lengths $2$, $-3$ and $-2$ respectively.
    \label{fig:H_trees}}
\end{figure}

\begin{definition}
    A \emph{H-tree} of degree $n>0$ and base-length $p\in\Z$ is a partition tree $(T,v_0,f_V,f_E)$, such that 
    \begin{enumerate}[label=(H.\arabic*),start=3]
        \item\label{enumitem:H_fE_leaf} $f_E(e)=0$ if and only if $e$ is incident to a leaf of the tree;
        \item\label{enumitem:H_fV} If $v$ a top vertex ($\ind_{f_E(v^-)\leq0}\geq\ind_{f_E(v^\textrm{L})\leq0}+\ind_{f_E(v^\textrm{R})\leq0}$), then $f_V(v)=0$.
        Heuristically, we can only have excess when a new non-positive child is created.
    \end{enumerate}
\end{definition}
The set of H-trees of degree $n$ and base-length $p$ is denoted by $\mathcal{H}_{n,p}$.

It should not be hard to see that the requirements of the labels match the rules for allowed signatures of subdisks \eqref{eq:signatures}.

\begin{theorem}\label{thm:bijection_rigid_quad_H_tree}
    There is a bijection between base-$p$ rigid quadrangulations with $2n$ corners and H-trees of degree $n$ and base-length $p$.
\end{theorem}
\begin{proof}
    \begin{figure}[p]
        \centering
        \includegraphics[width=.9\linewidth]{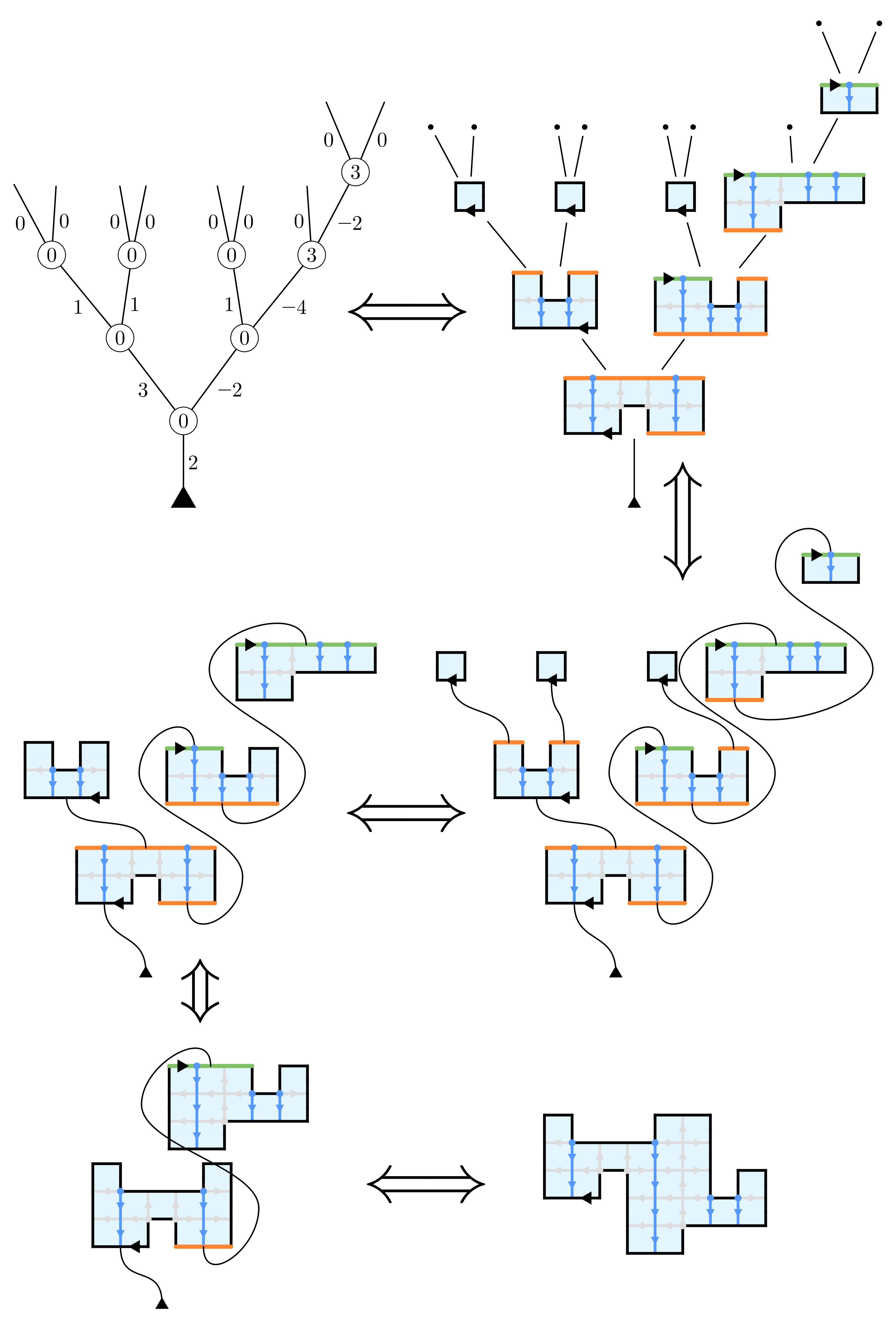}
        \caption{Example of the bijection between H-trees and based rigid quadrangulations.
        \label{fig:H_tree_to_disk}}
    \end{figure}

    As a consequence of Lemma \ref{lem:minimal_signed_submaps} and subsequent remark, a (complete) rigid quadrangulation of base-length $p\neq0$ with $2n$ corners is completely determined by its base-length $p$ and the two (ordered, complete) rigid quadrangulations glued to its unique minimal submap.\footnote{Here, we consider a convex corner as a complete rigid quadrangulation with base-length $0$.}
    These two quadrangulations are again determined by the base-length and two glued quadrangulations, and this can be repeated recursively until the quadrangulations all have base-length $0$, which are unique. 
    What remains are the base-lengths of all intermediate steps and how they are split, which can be represented by a rooted binary tree where the intermediate base-lengths are labels on the edges. 
    Each step in the decomposition then corresponds to an internal vertex, which we can label by $k=a+b-p+1$, which also directly corresponds to the number of downwards rays ending in the horizontal side in the middle (plus one for each bottom vertex in degenerate cases).
    It is straightforward to check that this tree is an H-tree of degree $n$ and base-length $p$.

    Conversely, given an H-tree we can assign to each leaf a base-$0$ rigid quadrangulation (a convex corner) and recursively glue the rigid quadrangulations to the minimal submaps corresponding to the signatures that can be read off from the internal vertices. 
    The properties of the H-tree ensures that these signatures are always allowed, so that a minimal submap exists.
\end{proof}
See Figure~\ref{fig:H_tree_to_disk} for an example.

\section{Q-trees}\label{sec:Q_trees}
Due to the non-trivial constraints on the labels, H-trees are hard to enumerate directly. 
By doing some clever shifts of the labels, we can relate the H-trees to similar labeled trees (Q-trees), but with different constraints, which are easier to count.

Let us first define a pre-Q-tree.
\begin{definition}
    A \emph{pre-Q-tree} of degree $n>0$ and base-length $p\in\Z$ is a partition tree $(T,v_0,f_V,f_E)$, such that 
    \begin{enumerate}[label=(Q.\arabic*),start=3]
        \item\label{enumitem:preQ_fe} If an edge $e$ is incident to a leaf, $f_E(e)\leq 0$ ;
        \item\label{enumitem:preQ_fV} $f_V(v)=0$ for all internal vertices $v$, so no more excess.
    \end{enumerate}
\end{definition}
The set of pre-Q-trees of degree $n$ and base-length $p$ is denoted by $\hat{\mathcal{Q}}_{n,p}$.

The condition \ref{enumitem:preQ_fV} is much easier than \ref{enumitem:H_fV}, which motivates the shift from H-trees to Q-trees.
Of course, there will be a price to pay.
\\\\
\begin{figure}[t]
    \centering
    \includegraphics[width=\linewidth]{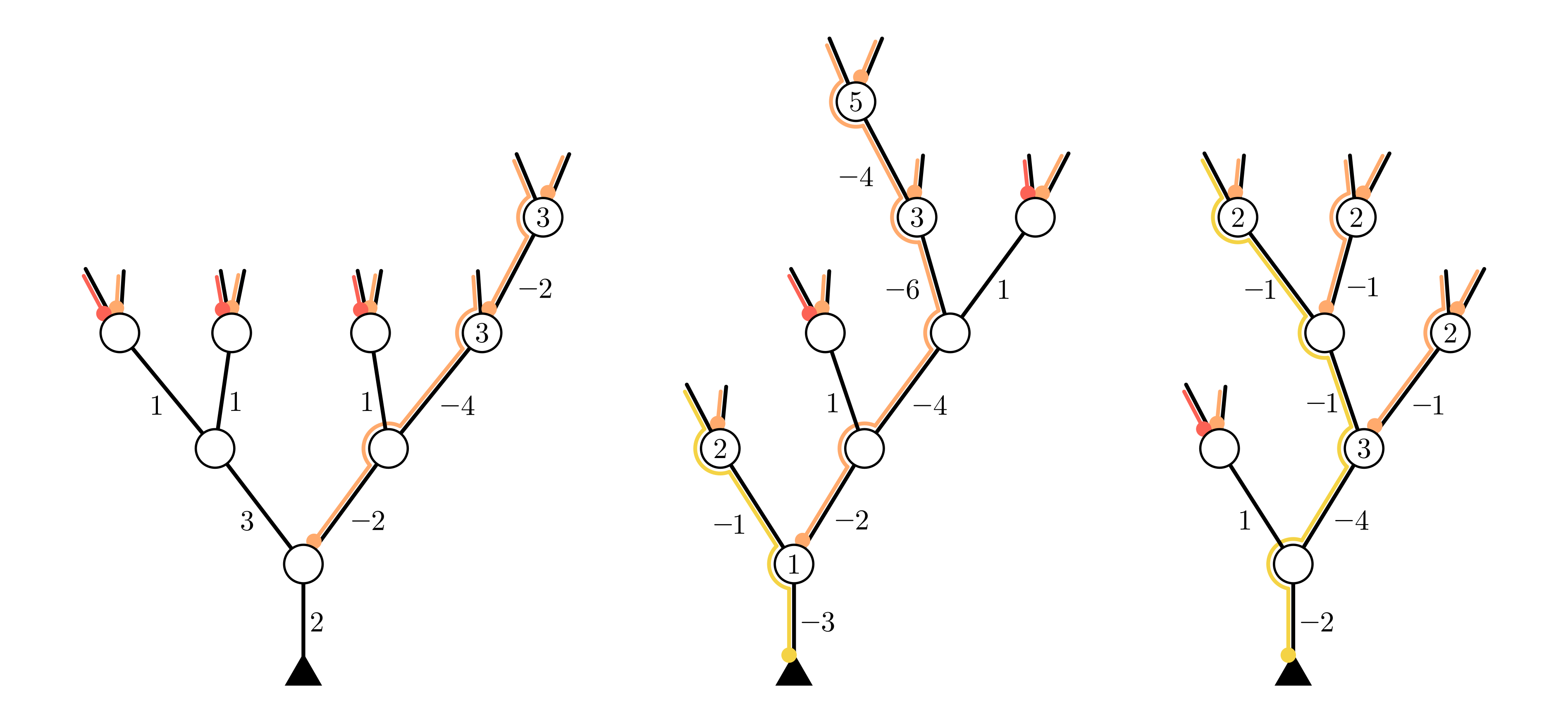}
    \caption{The H-trees from Figure~\ref{fig:H_trees} with the paths drawn on the trees. 
    The orange lines are the regular paths, the yellow lines are the root paths and the red lines are the special paths.
    \\
    For simplicity, we have omitted the (edge or vertex) label when the label is $0$.
    \\
    Note that special paths only occur when an edge with label~$1$ splits into two edges with label~$0$. 
    \label{fig:Paths}}
\end{figure}%
To do the shifts of the labels that will help us, we create paths in both H-trees and pre-Q-trees, consisting of consecutive non-positive edges.
We will have three types of paths:
\begin{enumerate}
    \item For each bottom vertex $v$, we have a regular path $P_\textrm{reg}(v)$. 
    The path starts by taking the rightmost non-positive child edge of $v$ and after that we repeatedly take the leftmost non-positive child edge until we hit a leaf. We will denote the final edge by $l(v)$.
    \item If the base-length is non-positive, we will also have a root path $P_0(v_0)$ starting at the root following the same steps as for the regular paths.
    \item There can be bottom vertices $v$ that have a positive parent edge and two non-positive child edges. 
    Besides the regular path $P_\textrm{reg}(v)$, we also create a special path $P_\textrm{spe}(v)$, which starts by taking the \emph{left} child edge, after which it follows the regular rules. 
\end{enumerate}
Note that the construction of paths does not depend on the exact values of the edge-labels, only on whether they are positive.

\begin{lemma}
    The construction of paths is well-defined for H-trees and pre-Q-trees.
\end{lemma}
\begin{proof}
    It follows from the definition of a bottom vertex that it has at least one non-positive child edge, so there exists a rightmost (leftmost) non-positive child edge.
    Furthermore, for any internal vertex $v$ in both H-trees and pre-Q-trees, we either have $f_V(v)=0$ or $v$ is a bottom vertex.
    This means that when we have a non-positive parent edge, it is guaranteed that at least one of the child edges is non-positive, so a leftmost non-positive child edge exists.  
\end{proof}
\begin{lemma}
    The paths in H-trees and pre-Q-trees have the following properties:
    \begin{enumerate}[nosep]
        \item Different paths are disjoint.
        \item All non-positive edges are in a path.
        \item All leaves are endpoints of paths.
        \item Special paths only start in vertices with label $0$, where the parent edge has label $1$ and both child edges have label $0$.   
    \end{enumerate}
\end{lemma}
\begin{proof}
    We prove the statements one by one.
    \begin{enumerate}
        \item Paths move away from the root. 
        Therefore, paths might only overlap when one is created `on top of' an existing path. 
        However, when we have an existing path and thus a non-positive parent edge, by the definition of a bottom vertex, both child edges are non-positive, so the leftmost non-positive child edge does not equal the rightmost.
        \item For a non-positive edge, it is either part of a path that is created in its parent vertex, or it is the left of two non-positive child edges with a non-negative parent edge, where it will be in the same path as this parent edge, if this parent edge is part of a path.
        In the latter case, we can recursively apply this argument to the non-negative parent edge. 
        This recursion has to stop, since we will ultimately reach the root edge, which is part of the root path when it is non-positive.
        We can thus conclude that all non-positive edges are in a path.
        \item This immediately follows from the previous property, combined with the fact that edges incident to leaves are by definition non-positive.
        \item In this case, we have $f_E(v^-)>0$, $f_E(v^\textrm{L})\leq0$, $f_E(v^\textrm{R})\leq0$ and $f_V(v)\geq0$. 
        The only solution to $f_V(v)=f_E(v^\textrm{L})+f_E(v^\textrm{R}) - f_E(v^-)+1$ is $f_E(v^-)=1$ and $f_E(v^\textrm{L})=f_E(v^\textrm{R})=f_V(v)=0$.
    \end{enumerate}
\end{proof}

A path is called \emph{strong} when the edge labels along the path attain a unique maximum at the last edge of the path (the edge incident to a leaf). 
It follows directly from the requirement \ref{enumitem:H_fE_leaf} that all paths in an H-tree are strong.

\begin{definition}
    A \emph{Q-tree} is a pre-Q-tree such that
    \begin{enumerate}[label=(Q.\arabic*),start=5]
        \item\label{enumitem:Q_dom} all paths are strong.
    \end{enumerate}
\end{definition}
The set of Q-trees of degree $n$ and base-length $p$ is denoted by $\mathcal{Q}_{n,p}\subseteq\hat{\mathcal{Q}}_{n,p}$.
\\\\
Note that in both H-trees and Q-trees the special paths will consist of only a single edge, since an edge with label $0$ must be incident to a leaf.

A Q-tree is called \emph{well-based}, when either it has positive base-length or when the root path ends on an edge with label $0$.
The set of well-based Q-trees of degree $n$ and base-length $p$ is denoted by $\mathcal{Q^*}_{n,p}\subseteq\mathcal{Q}_{n,p}$.

\begin{figure}
    \centering
    \includegraphics[width=\linewidth]{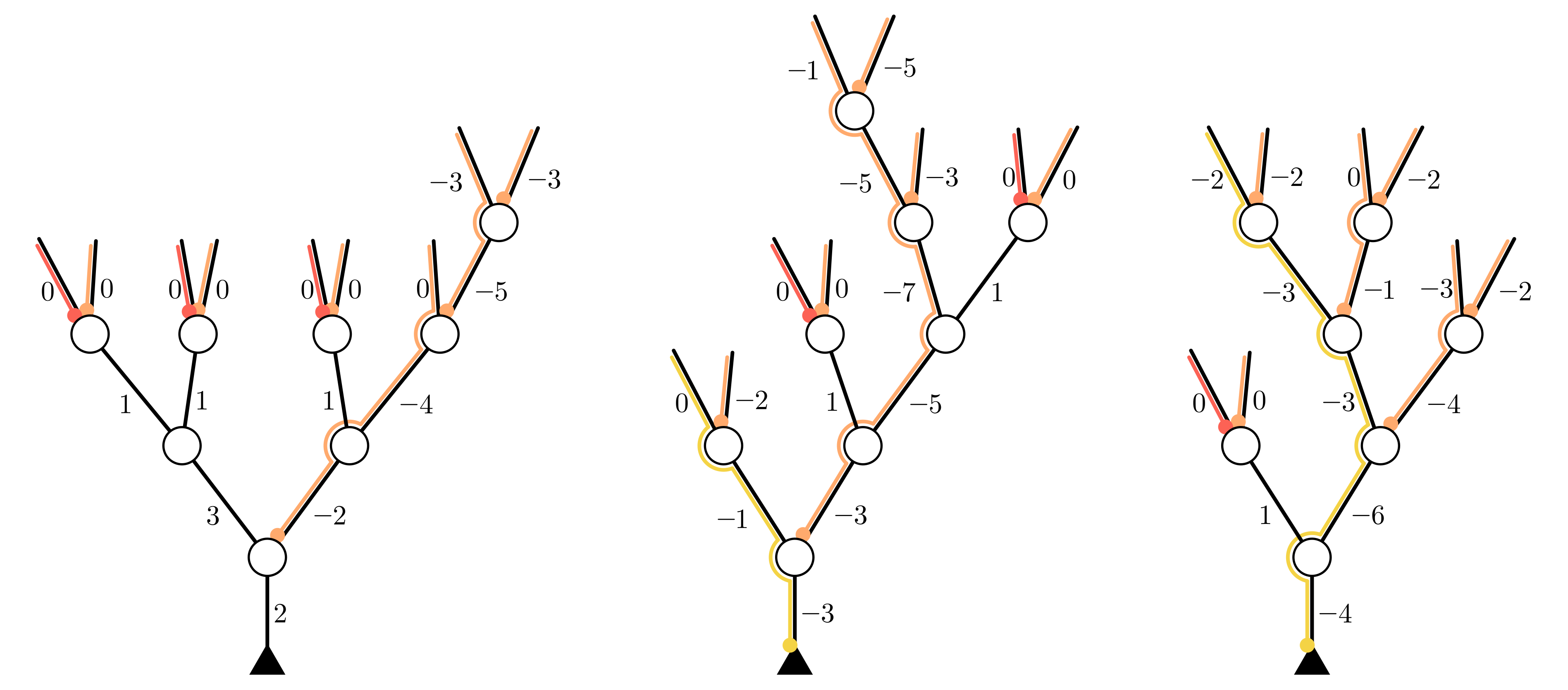}
    \caption{Some examples of Q-trees.
    The vertex labels have been omitted, since they vanish by definition. 
    Note that all paths are strong, meaning that the edge label at the leaf has the unique highest value within the path.
    \\
    The first two examples are well-based Q-trees. 
    In fact, they are the result of applying $\psi$ to the first two examples of H-trees in Figures~\ref{fig:H_trees} and \ref{fig:Paths}. 
    The last example is not well-based: the yellow root path ends on an edge with a non-zero label ($-2$). 
    It is the result of applying $\hat{\psi}_{-4}$ to the third example of H-trees in Figures~\ref{fig:H_trees} and \ref{fig:Paths}.
    \label{fig:Q_trees}}
\end{figure}

\subsection*{Relations between H-trees and Q-trees}

The H- and Q-trees are related. In fact, there is a bijection between H-trees and well-based Q-trees. 

\begin{proposition}\label{prop:H_Q*_bij}
    For all $n,p$ there exists a bijection $\psi$ between $\mathcal{H}_{n,p}$ and $\mathcal{Q}^*_{n,p}$.
\end{proposition}

\begin{proof} 
    We will have to get rid of the $f_V(v)\neq0$ cases that can occur when $v$ is a bottom vertex, while keeping requirement~\ref{enumitem:sum_condition}. 
    We can do this by lowering the labels on the path that starts at $v$.  

    To be precise: 
    Let $\psi$ be the map from H-trees to well-based Q-trees by keeping the rooted tree and sending $(f_V,f_E)\mapsto (f_V',f_E')$ such that $f_V'=0$ and
    \begin{align}
        f_E'(e)=\begin{dcases}
            f_E(e) - f_V(v) &\textrm{ if } e\in P_\textrm{reg}(v)\\
            f_E(e)          &\textrm{ otherwise}
        \end{dcases}.
    \end{align}
    Note that (non-)positive edges remain (non-)positive, so that the paths are invariant under $\psi$. 
    Furthermore, strong paths remain strong, since all edge labels within a path are shifted by the same amount.
    
    It is therefore easy to see that the $\psi$-image of an H-tree is indeed a well-based Q-tree of the same base-length and degree.

    We can find the inverse $\psi^{-1}$ that maps well-based Q-trees to H-trees by, again, keeping the rooted tree and sending $(f_V,f_E)\mapsto (f_V',f_E')$ such that 
    \begin{align}
        f_V'(v)=\begin{dcases}
            -f_E(l(v)) & \textrm{ if $v$ a bottom vertex}\\  
            0 & \textrm{ otherwise}
        \end{dcases}
    \end{align}
    and
    \begin{align}
        f_E'(e)=\begin{dcases}
            f_E(e) - f_E(l(v)) &\textrm{ if } e\in P_\textrm{reg}(v)\\
            f_E(e)          &\textrm{ otherwise}
        \end{dcases},
    \end{align}
    where we remind the reader that $l(v)$ is the final edge in the path $P_\textrm{reg}(v)$.
    Also note that since the paths are strong, the paths are again preserved, and we have $f_E(e)=0$ only when $e$ is incident to a leaf.
\end{proof}

For $p<0$ there exist Q-trees that are not well-based, for which the root path ends in an edge with a strictly negative label.
To include these, we can also find a bijection for Q-trees in general.

\begin{proposition}
    For all $p<0$ and all $n$ there exists a bijection
    \begin{align}
        \hat{\psi}_p: \bigcup_{p\leq p' \leq0}\mathcal{H}_{n,p'} \rightarrow \mathcal{Q}_{n,p}.
    \end{align}
\end{proposition}
\begin{proof}
    The map $\hat{\psi}_p$ and its inverse are very similar to $\psi$ and its inverse. 
    The only difference occurs in $\hat{\psi}_p$ when an edge is in the root path $P_0(v_0)$.
    In this case for $\hat{\psi}_p$, we set 
    \begin{align}
        f_E'(e)=f_E'(e)-p'+p &\textrm{ if }e\in P_0(v_0),
    \end{align}
    where $p'$ is the base-length of the H-tree. 
    Note that $p\leq p'$, thus the label is lowered.
    All other mappings are the same as for $\psi$. 

    Similarly, for $\hat{\psi}^{-1}_p$, we set 
    \begin{align}
        f_E'(e)=f_E(e) - f_E(l(v_0)) \textrm{ if }e\in P_0(v_0),
    \end{align}
    where $l(v_0)$ is the last edge in the path $P_0(v_0)$.
    The strong path condition \ref{enumitem:Q_dom} ensures that the label of the edge incident to the root in the image is between $p$ and $0$. 
\end{proof}

\FloatBarrier
\section{Enumeration}\label{sec:enumeration}
\FloatBarrier

To enumerate the various trees, let's define the base-$p$ generating functions for H-trees, pre-Q-trees and Q-trees as
\begin{align}
    H^{(p)}(t)=\sum_{n\geq1} |\mathcal{H}_{n,p}| t^n,\\
    \hat{Q}^{(p)}(t)=\sum_{n\geq1} |\hat{\mathcal{Q}}_{n,p}| t^n,\\
    Q^{(p)}(t)=\sum_{n\geq1} |\mathcal{Q}_{n,p}| t^n.
\end{align}

\begin{proposition}\label{prop:gen_Q_tree}
    As before, let $R(t)\in\Z[[t]]$ be the unique formal power series solution to
    \begin{align}
        \sum_{n\geq0}\frac{1}{n+1}\binom{2n}{n}\binom{2n}{n} R(t)^{n+1} = t,
    \end{align}
    with $[t^0]R(t)=0$.

    Then we have
    \begin{align}
        Q^{(p)}(t)=\sum_{n\geq \max\{0,p\}}\frac{1}{n+1}\binom{2n}{n}\binom{2n-p}{n-p} R(t)^{n+1}.
    \end{align}
\end{proposition}
From this result, the generating function of H-trees follows directly:
\begin{corollary}\label{coro:H_gen_fun}
    Let $R(t)$ be as above. 
    Then we have
    \begin{align}
        H^{(p)}(t)=\begin{dcases}
            \sum_{n\geq p}\frac{1}{n+1}\binom{2n}{n}\binom{2n-p}{n-p} R(t)^{n+1} & \textrm{ if } p\geq0, \\
            \sum_{n\geq 1}\frac{1}{n+1}\binom{2n}{n}\binom{2n-p-1}{n-p} R(t)^{n+1} & \textrm{ if } p<0.
        \end{dcases}
    \end{align}
\end{corollary}
\begin{proof}
    The first case follows directly from the bijection between H-trees and well-based Q-trees and the fact the all Q-trees are well-based for $p\geq0$.
    The second case follows from the bijection between unions of H-trees and Q-trees, which gives $H^{(p)}(t)=Q^{(p)}(t)-Q^{(p+1)}(t)$ for $p<0$.
\end{proof}

\noindent To prove the proposition, let's first have a look at the pre-Q-trees.
Note that they are easily counted:
\begin{lemma}\label{lem:Qhat_gen}
    The number of pre-Q-trees of degree $n+1$ and base-length $p$ are
    \begin{align}
        |\hat{\mathcal{Q}}_{n+1,p}|=\frac{1}{n+1}\binom{2n}{n}\binom{2n-p}{n-p},
    \end{align}
    which gives
    \begin{align}
        \hat{Q}^{(p)}(t)=\sum_{n\geq0}\frac{1}{n+1}\binom{2n}{n}\binom{2n-p}{n-p}t^{n+1}.
    \end{align}
\end{lemma}
\begin{proof}
    Note that the labels of a pre-Q-tree of degree $n+1$ and base-length $p$ are completely determined by the labels on the edges incident to the leaves. 
    These labels are non-positive and due to conditions \ref{enumitem:sum_condition}, \ref{enumitem:root_label} and \ref{enumitem:preQ_fV}, their sum is equal to $p-n$.
    Apart from these restrictions all combinations of these labels are valid. 
    Since the number of rooted binary plane trees with $n+1$ leaves are given by the Catalan number $\frac{1}{n+1}\binom{2n}{n}$ and the number of ways to partition $p-n$ into $n+1$ non-positive numbers is $\binom{2n-p}{n-p}$, the lemma follows. 
\end{proof}

Now let's study the relation between pre-Q-trees and Q-trees. 
We will use the fact that every pre-Q-tree has a maximal Q-tree `inside'.
\begin{figure}
    \centering
    \includegraphics[width=\linewidth]{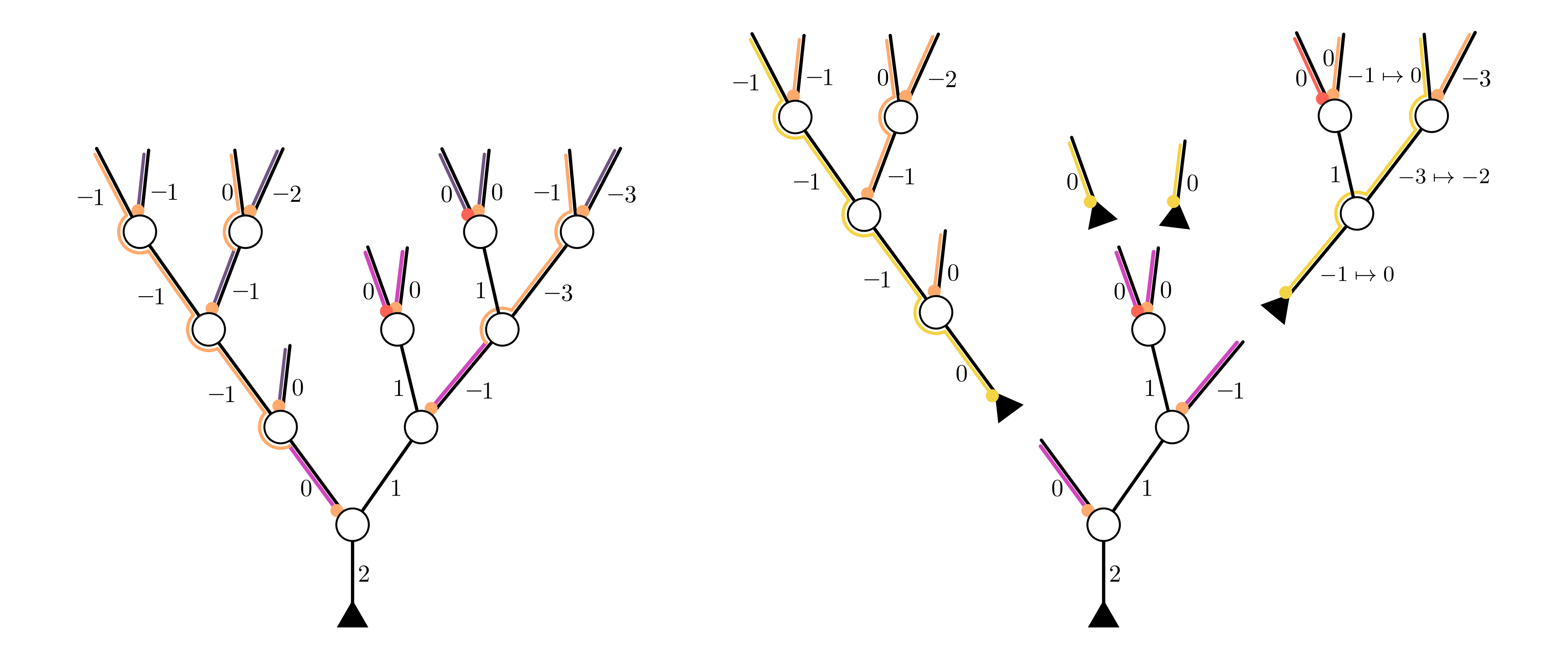}
    \caption{On the left is an example of a pre-Q-tree $q$ which is not a Q-tree, since there are 3 paths that are not strong. 
    The strongest edges that can be reached from the root without traversing other strong edges, thus the set $D(q)$, are colored light purple, while the other strong edges are colored dark purple.
    \\
    The decomposition map $\tilde{\phi}$ lets us `cut' at the light purple edges, giving the picture on the right. 
    The rightmost removed subtree has non-zero label ($-1$) on the root, so needs to be rebased, by increasing the labels along the yellow root path by $1$.
    \\
    The resulting removed subtrees are now all base-$0$ pre-Q-trees. 
    \label{fig:PreQQ_map}}
\end{figure}
To be more precise, let's define the \emph{truncation of a pre-Q-tree at non-positive edge $e$} to be the pre-Q-tree that remains after we contract everything above $e$ (the side that doesn't include the root) into a single leaf vertex.
The corresponding \emph{removed subtree} is pre-Q-tree that remains when we contract everything below $e$ (the side that does include the root) into a single root vertex. 

For a pre-Q-tree $q$ we can find for each path $P(v)$ the \emph{strongest edge}, which is the edge closest to the root for which $f_E$ is maximal.
Note that a path is strong if and only if the final edge is the strongest edge.
Let $D(q)$ be the set of strongest edges that can be reached from the root without traversing another strongest edge.

We can now do a decomposition $\tilde{\phi}: q\mapsto(q_0,\tilde{q}_1,\dots,\tilde{q}_m)$, where $m=|D(q)|$, $q_0$ is the resulting tree that we get after truncating at all edges in $D(q)$ and $\tilde{q}_i$ the removed subtree corresponding to the $i$-th element of $D(q)$.

To each removed subtree $\tilde{q}_i$, we can apply a \emph{rebasing map} $\tilde{q}_i\mapsto q_i$, which lowers the edge labels of all edges in the root path by $f_E(e_0)$, where $e_0$ is the edge incident to the root. 
Note that since $e_0$ is the strongest edge in the path, the rebased removed subtree $q_i$ is still a pre-Q-tree and has the same paths as $\tilde{q}_i$, but always has base-length $0$.

Now consider the decomposition $\phi$ which is simply $\tilde{\phi}$, where we rebase all removed subtrees. See Figure~\ref{fig:PreQQ_map} for an example.

\begin{lemma}\label{lem:decomp}
    Let $q\in\hat{Q}_{n,p}$ be a pre-Q-tree of degree $n$ and base-length $p$ and $\phi(q)=(q_0,q_1,\dots,q_m)$ its decomposition.
    We have $q_0\in Q_{m,p}$ and $q_i\in\hat{Q}_{n_i,0}$ for $i>0$, such that $\sum_{i=1}^m n_i=n$.

    Furthermore, $\phi$ seen as the map
    \begin{align}
        \phi: \hat{Q}_{n,p} \rightarrow \Big\{(q_0,q_1,\dots,q_m): 0<m\leq n,\,\, q_0\in Q_{m,p},\,\, q_i\in\hat{Q}_{n_i,0}\,\,\forall\,\,i>0,\,\, \sum_{i=1}^m n_i=n\Big\}
    \end{align} 
    is a bijection.
\end{lemma}
\begin{proof}
    Since we contract all subtrees above the strongest edges, it is clear that all paths of $q_0$ are strong, so it is a Q-tree.
    It has still base-length $p$ and has $m$ leaves. 
    The original $n$ leaves are distributed over the $q_i$'s, so we have $\sum_{i=1}^m n_i=n$.

    Now we need to check that this is a bijection. 
    Given $(q_0,q_1,\dots,q_m)$, we can use the labels on the $m$ edges incident to the leaves of $q_0$ to invert the rebasing maps, since the paths are invariant under the rebasing. 
    Next, we can simply `re-glue' the edges incident to leaves of $q_0$ to the corresponding trees $\tilde{q}_i$, inverting $\phi$. 
    
    This reverse process clearly gives a pre-Q-tree and is an inverse of $\phi$, so $\phi$ is a bijection.
\end{proof}

\begin{proof}[Proof of proposition \ref{prop:gen_Q_tree}]
    From the decomposition bijection, we get
    \begin{align}
        \hat{Q}^{(p)}(t)=Q^{(p)}(\hat{Q}^{(0)}(t)).
    \end{align}
    By definition, we have $\hat{Q}^{(0)}(R(t))=t$, we get
    \begin{align}
        Q^{(p)}(t)=\hat{Q}^{(p)}(R(t)),
    \end{align}
    which, together with lemma \ref{lem:Qhat_gen} proves the proposition.
\end{proof}

Theorem \ref{thm:gen_fun_rigid} now follows immediately from Theorem \ref{thm:bijection_rigid_quad_H_tree} and Corollary \ref{coro:H_gen_fun}.

\FloatBarrier
\section{B-, C- and \texorpdfstring{$\Delta$}{Delta}-type rigid quadrangulations}\label{sec:BCD}
\FloatBarrier

\begin{figure}
    \centering
    \includegraphics[width=\linewidth]{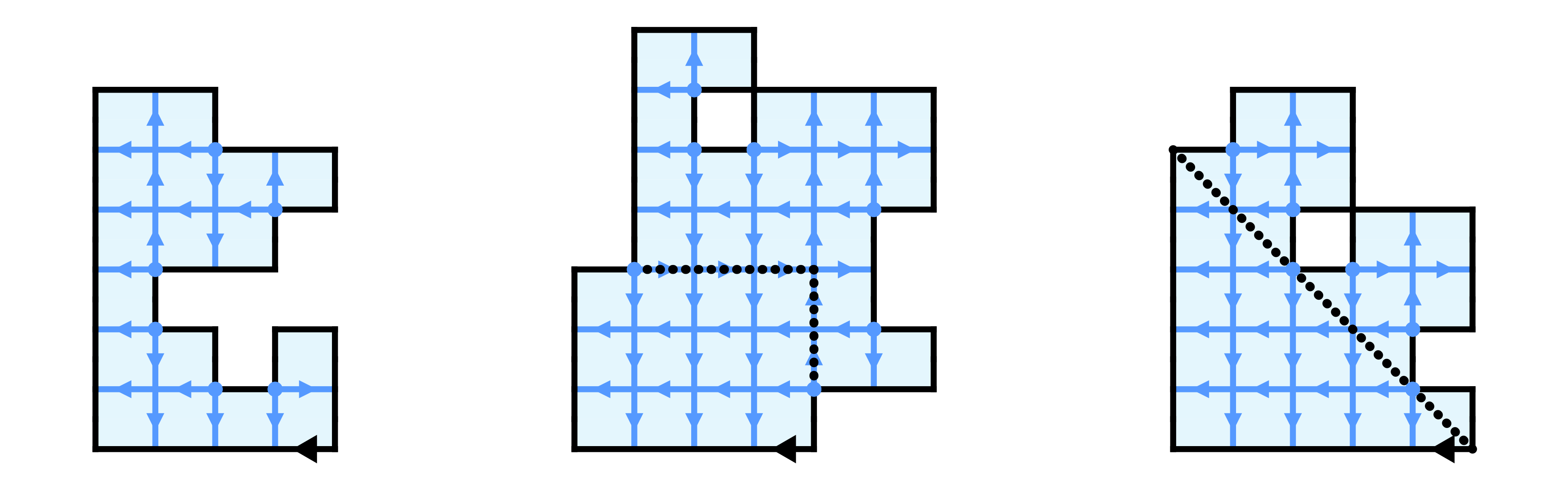}
    \caption{A B-type, a C-type and a $\Delta$-type rigid quadrangulation.
    The B-type rigid quadrangulation has base-length $4$ and co-base-length $6$.
    The C-type rigid quadrangulation has base-length $4$ and co-base-length $3$.
    The $\Delta$-type rigid quadrangulation has base-length $5$ and co-base-length $5$ and degeneracy $3$.
    \label{fig:BCD_2}}
\end{figure}

We like to remind the reader that a B-type rigid quadrangulation is a double based rigid quadrangulation, with no further requirements, and C-type rigid quadrangulations and $\Delta$-type rigid quadrangulations are B-type rigid quadrangulations with an embedded rectangle or triangle, respectively. 
See Figure~\ref{fig:BCD_2}.
Recall that the \emph{degeneracy} of a $\Delta$-type rigid quadrangulation is the number of concave corners on the hypotenuse, plus one.

The generating functions of B- and C-type rigid quadrangulations are equal to the generating functions of B- and C-patches, due to \cite{Budd_rectilinear_2025}, which have been computed by \cite{Bousquet-Melou_generating_2020}.
In the latter work, the generating function of $\Delta$-type rigid quadrangulations appear as a natural building block, but has not been given an interpretation there (nor in \cite{Budd_rectilinear_2025}).

In this section, we will show that the generating function of $\Delta$-type rigid quadrangulations can be derived using the tree bijection central to this paper. 
Furthermore, we will relate B- and C-type rigid quadrangulations to these $\Delta$-type rigid quadrangulations, giving the desired expressions for their generating functions.

\subsection{Corresponding Q-trees}
We will look at our bijection from rigid quadrangulations to Q-trees and restrict it to our B-, C- and $\Delta$-type rigid quadrangulations. 
Only the restriction to $\Delta$-type rigid quadrangulations is needed to prove the generating functions, but the other restrictions are added for completeness.

\begin{figure}
    \centering
    \includegraphics[width=\linewidth]{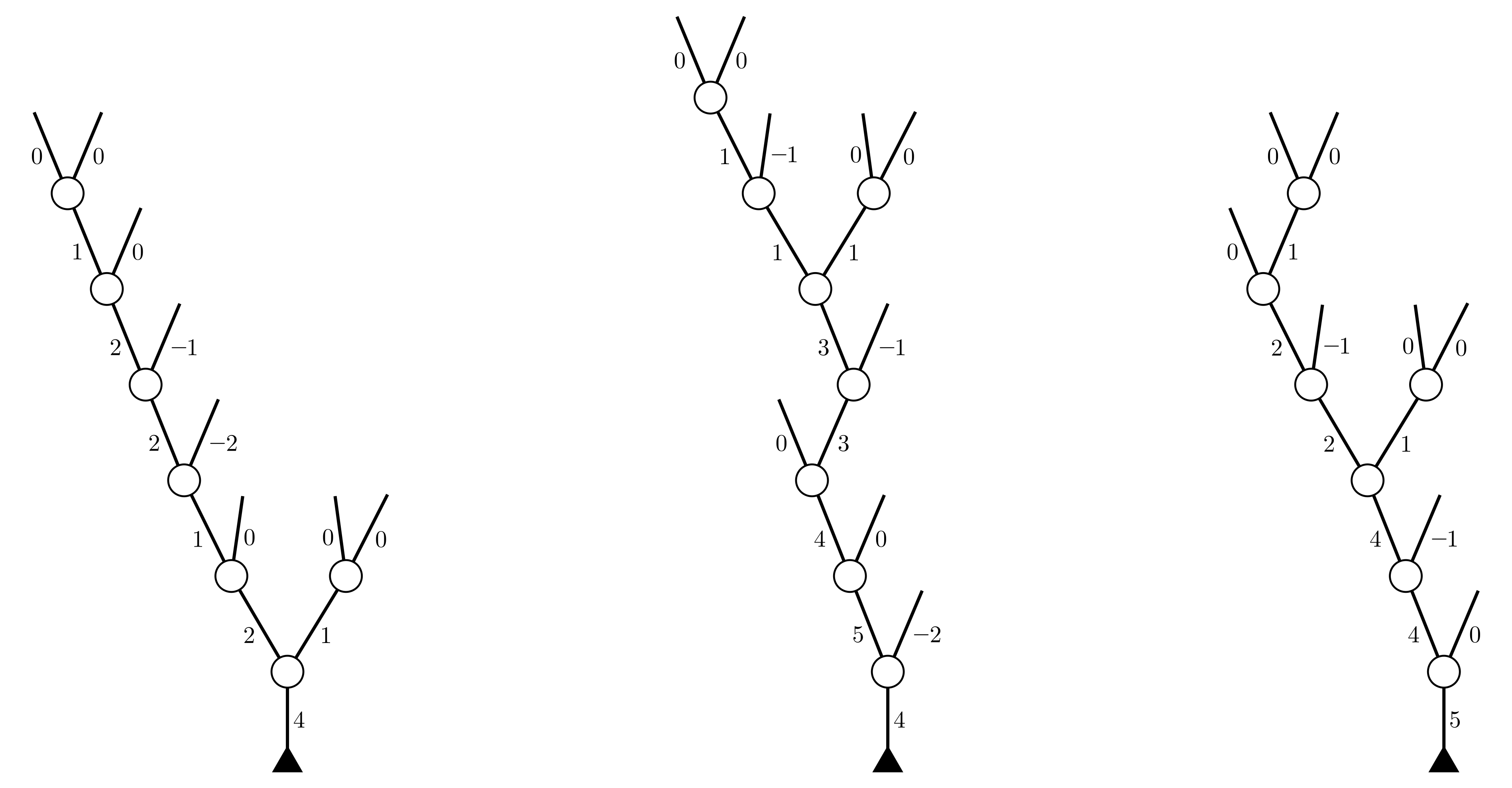}
    \caption{The Q-trees corresponding to the examples in Figure~\ref{fig:BCD_2}.
    It is easy to check that they have the correct form for a B-, C- or $\Delta$-type rigid quadrangulation respectively.
    \label{fig:BCD_tree_spec}}
\end{figure}

\begin{proposition}\label{prop:BCD_disk_to_Qtree}
The image of the set of B-type rigid quadrangulations with base-length $p$ and co-base-length $q$ under the tree bijection from Theorem~\ref{thm:bijection_rigid_quad_H_tree} composed with $\psi$ from Proposition~\ref{prop:H_Q*_bij}, is a subset of Q-trees, for which:
\begin{itemize}[nosep]
    \item The root has label $p$;
    \item Starting from the root edge, there are exactly $q$ consecutive left child edges, called co-base edges, with all strictly positive labels $a_i>0$, except for the final edge, which has label $a_q=0$.
\end{itemize}

\noindent The image of the set of C-type rigid quadrangulations with base-length $p$ and co-base-length $q$ under the tree bijection from Theorem~\ref{thm:bijection_rigid_quad_H_tree} composed with $\psi$ from Proposition~\ref{prop:H_Q*_bij}, is a subset of Q-trees, for which:
\begin{itemize}[nosep]
    \item The root has label $p$;
    \item Starting from the root edge, there are exactly $q$ consecutive left child edges, called co-base edges, with all labels $a_i\geq p$, except for the final edge, which has label $a_q=0$.
\end{itemize}

\noindent The image of the set of $\Delta$-type rigid quadrangulations with base-length $p$ and co-base-length $q$ under the tree bijection from Theorem~\ref{thm:bijection_rigid_quad_H_tree} composed with $\psi$ from Proposition~\ref{prop:H_Q*_bij}, is a subset of Q-trees, for which:
\begin{itemize}[nosep]
    \item The root has label $p$;
    \item Starting from the root edge, there are exactly $q$ consecutive left child edges, called co-base edges, with the $i$-th co-base edge having a label of at least $a_i\geq p-i p/q$.\footnote{Note that together with requirement \ref{enumitem:preQ_fe} this implies that the final edge has label $a_q=0$.}
    \item The degeneracy of the $\Delta$-type rigid quadrangulations corresponds to the number of left child edges for which the lower bound on the label is exactly met.\footnote{Note that $q$-th edge always meets the bound.} 
\end{itemize}
See Figure~\ref{fig:BCD_tree_gen}
\end{proposition}

\begin{figure}
    \centering
    \includegraphics[width=\linewidth]{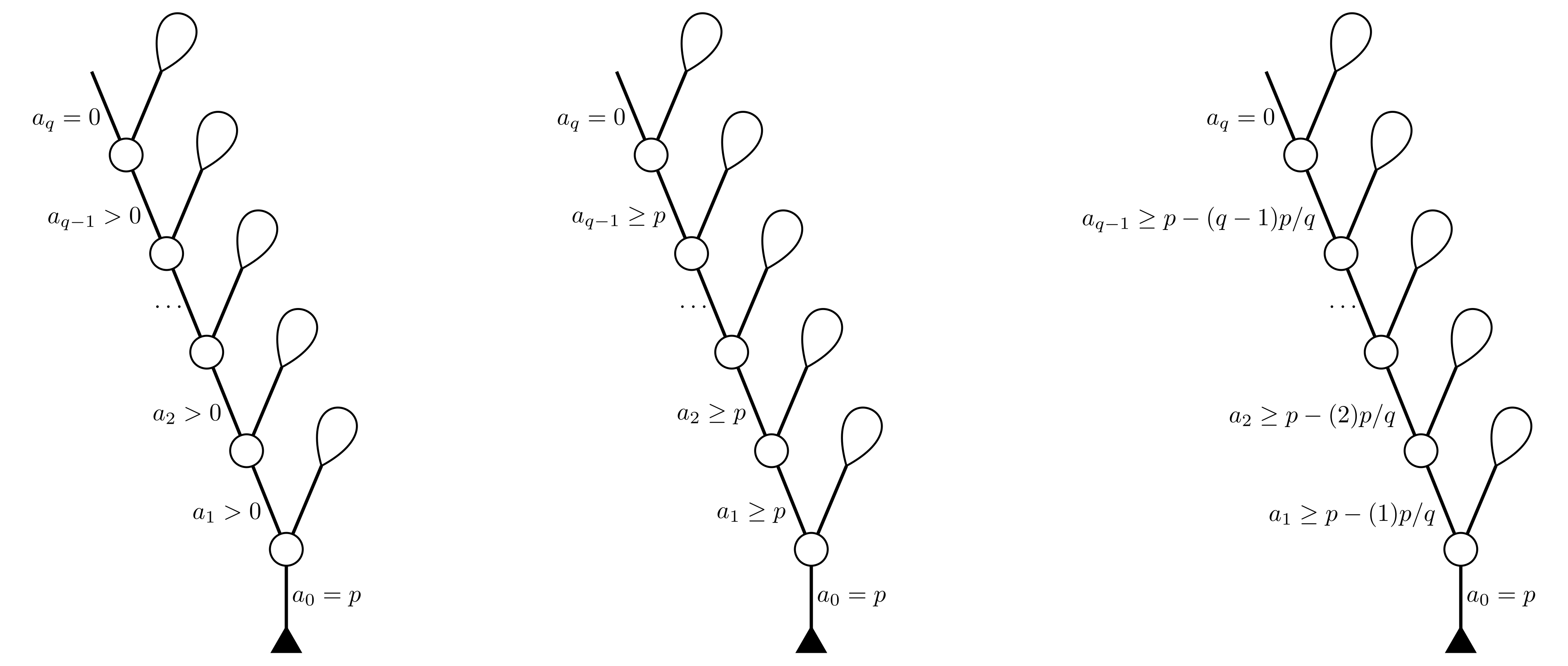}
    \caption{The general structure of the Q-tree corresponding to a B-, C- or $\Delta$-type rigid quadrangulation respectively.
    \label{fig:BCD_tree_gen}}
\end{figure}

\begin{proof}
    Let's start with the B-type rigid quadrangulations.
    By definition, we need $(q-1)$ incoming rays on the co-base, thus we need $(q-1)$ times (continuing to the left open side) a minimal submap of type $\mathbf{G}^{(0)}_{a_i,\bullet}$ with $a_i>0$ or $\mathbf{R}^{(\bullet)}_{a_i,\bullet}$.
    These correspond to all allowed signatures $(a_{i-1},a_i,b_i,k_i)$ with $a_{i-1}>0$ and $a_i>0$.
    After that we need to end the co-base in a convex corner, which means a submap of type $\mathbf{G}^{(0)}_{0,\bullet}$, or equivalently an allowed signature of the form $(a_{q-1}>0,a_q=0,b_q,k_q=0)$.
    
    Following the bijection from Theorem~\ref{thm:bijection_rigid_quad_H_tree}, we get an H-tree with root label $p$ and exactly $q$ co-base edges with positive labels $a_i>0$, except the final edge, that has label $a_q=0$ and starts at a vertex with label $k_q=0$. 

    Continuing to the Q-tree, we see that the labels of the root edge and the co-base edges are unchanged, since they are either positive or they correspond to a path that starts with label $0$.

    The argument also works in reverse: All Q-trees as described above map to H-trees with the same requirements. 
    Furthermore, the vertex at the start of the final consecutive left child of the root edge gets label $0$. 
    When mapping back to a rigid quadrangulation, we get $(q-1)$ times a submap of type  $\mathbf{G}^{(0)}_{a_i,\bullet}$ with $a_i>0$ or $\mathbf{R}^{(\bullet)}_{a_i,\bullet}$, and we end in a submap of type $\mathbf{G}^{(0)}_{0,\bullet}$. 
    Gluing these (and all other submaps), clearly gives a B-type rigid quadrangulation, with base-length $p$ and co-base-length $q$.
    \\
    \\
    The proof for C-type rigid quadrangulations and $\Delta$-type rigid quadrangulations is very similar. 
    With respect to B-type rigid quadrangulations, we get the extra requirements that the rays hitting the co-base have  some minimum lengths.
    
    In particular, for C-type rigid quadrangulations, the rays hitting the co-base should have lengths of at least $p$.
    We can now simply follow the proof for B-type rigid quadrangulations above, where we only need to require that $a_i\geq p$ for the labels of the $(q-1)$ first co-base edges.

    Similarly, for $\Delta$-type rigid quadrangulations the extra requirement is that the $i$-th incoming ray at the co-base has length of at least $p-i p/q$. 
    The proof is otherwise entirely analogous to the proof for B-type rigid quadrangulations.
\end{proof}

\subsection{Counting the trees corresponding to \texorpdfstring{$\Delta$}{Delta}-type rigid quadrangulations}
We are now ready to count the $\Delta$-type rigid quadrangulations, using their Q-tree-equivalents.

\begin{proposition}\label{prop:D_counting}
   The generating function of $\Delta$-type rigid quadrangulations as defined in Theorem~\ref{thm:D_B_C_gfs} is given by
    \begin{align}
        \Delta(t,x,y)=\sum_{n\geq0}\sum_{j=0}^n\sum_{i=0}^n\frac{1}{n+1}\binom{2n-i}{n}\binom{2n-j}{n}x^{i+1}y^{j+1}R^{n+1},
    \end{align} 
    where, as always, $R=R(t)$ is the solution of equation \eqref{eq:R_def}. 
\end{proposition}
\begin{proof}
    We start with analyzing $\Delta$-type rigid quadrangulations with fixed base-length $j+1$ and fixed co-base-length $i+1$.
    As before, we will call the labels on the root edge and the $i+1$ co-base edges $(a_0=j+1,a_1,\ldots,a_i,a_{i+1}=0)$.
    Note that we require 
    \begin{align}\label{eq:mid_D_req}
    a_k\geq (j+1)(1-k/(j+1)).
    \end{align}

    We will first look at pre-Q-trees that have the correct tree-structure, having $a_0=j+1$ and $a_{i+1}=0$, but \eqref{eq:mid_D_req} does not necessarily hold for $0<k\leq i$.

    Counting these pre-Q-trees is relatively straightforward.  
    Note that, upon removing the root edge and the co-base edges, we are left with $i+1$ separate subtrees, each counted by a Catalan number.
    Therefore, the number of tree structures with $n+1$ leaves (excluding the leaf at the end of the final co-base edge), is equivalent to the $i+1$ convolution 
    \begin{align}
        \sum_{n_1+\cdots+n_{i+1}=n-i} \frac{1}{n_1+1}\binom{2n_1}{n_1}\cdots \frac{1}{n_{i+1}+1}\binom{2n_{i+1}}{n_{i+1}} = \frac{i+1}{n+1}\binom{2n-i}{n}. 
    \end{align}  
    For the labels of the pre-Q-trees, we distribute $n-j$ over the $n+1$ leaves, giving $\binom{2n-j}{n}$ options.
    \\\\
    \begin{figure}[p]
        \centering
        \includegraphics[width=\linewidth]{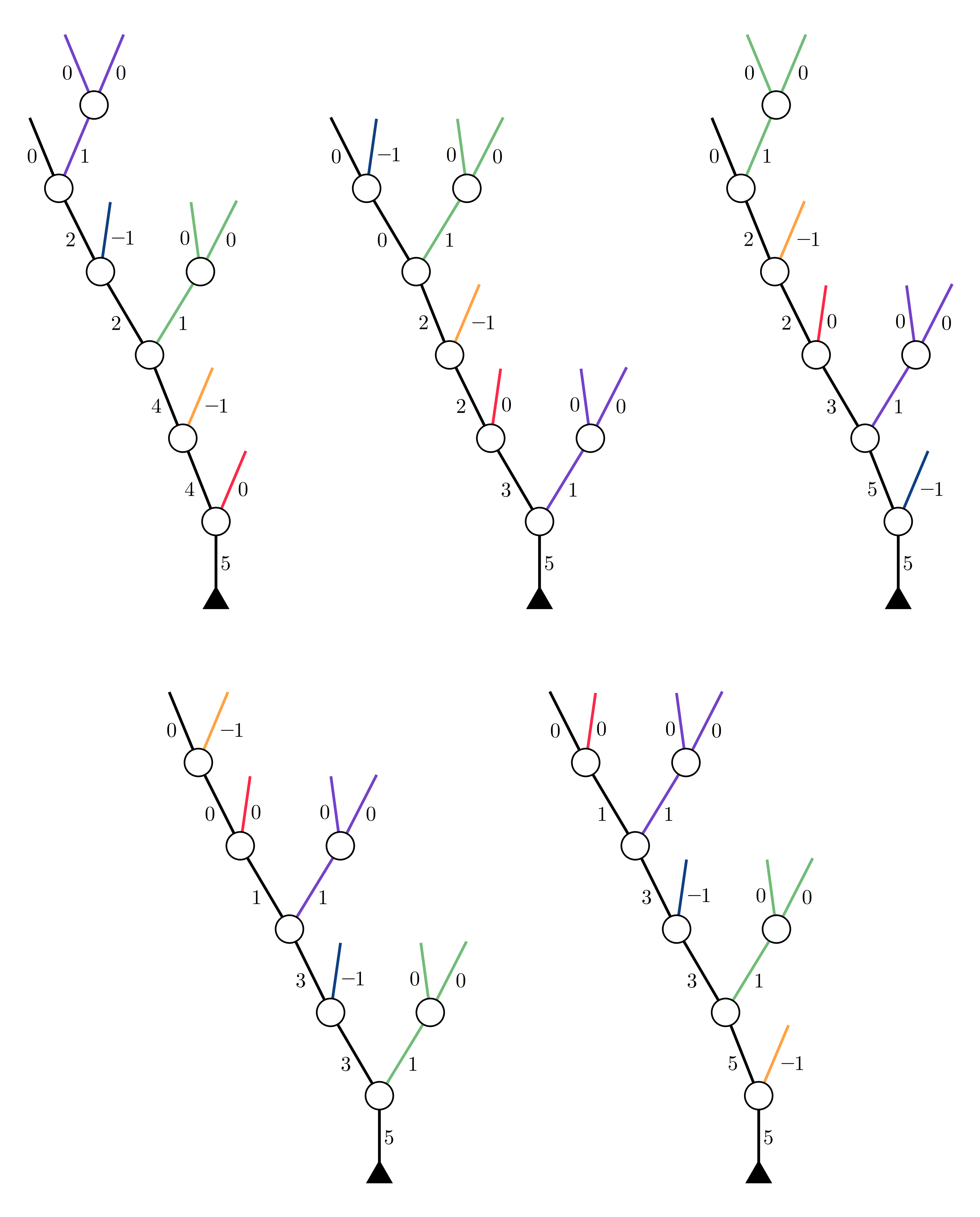}
        \caption{Cycling of subtrees to create a valid $\Delta$-tree. 
        Not that the first, third and fifth examples are valid.
        This agrees with the fact that the corresponding $\Delta$-type rigid quadrangulation has degeneracy~$3$.
        \label{fig:D_tree_cycle}}
    \end{figure}
    \begin{figure}
        \centering
        \includegraphics[width=\linewidth]{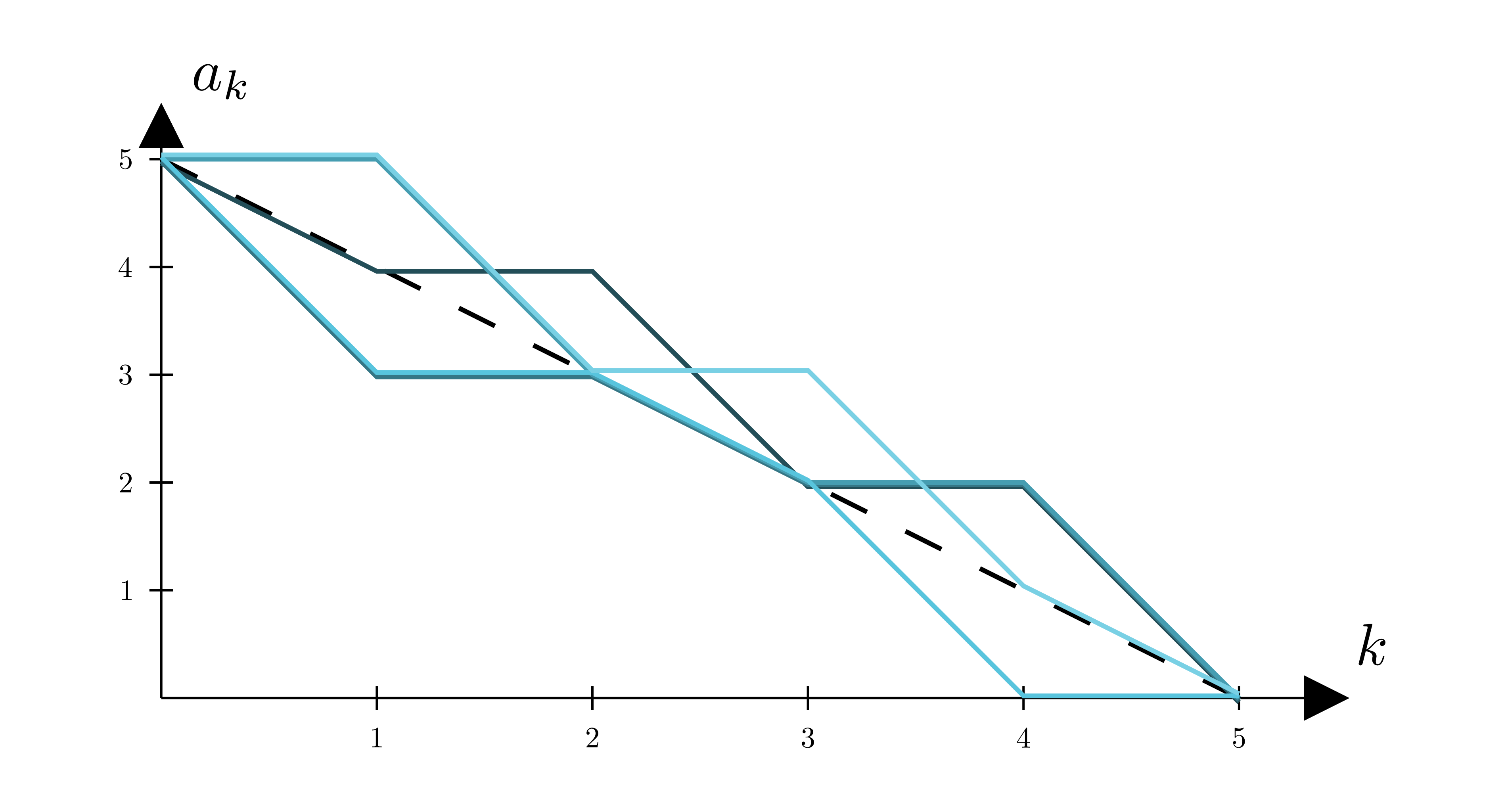}
        \caption{The labels $a_k$ can be seen as a walk from $a_0=p$ to $a_{i+1}=0$, where $i+1$ is the co-base-length. 
        Cycling subtrees corresponds to cycling the steps of this walk. 
        The $\Delta$-type rigid quadrangulation condition \eqref{eq:mid_D_req} corresponds to conditioning the walk to stay on or above the dashed line from start to end. 
        \label{fig:D_cycle_assist}}
    \end{figure}
    To satisfy \eqref{eq:mid_D_req}, we can cycle the subtrees. 
    If we have degeneracy $m$, we have $m$ ways to cycle the subtrees to satisfy \eqref{eq:mid_D_req}, see Figures~\ref{fig:D_tree_cycle} and \ref{fig:D_cycle_assist}, giving a $(i+1)$-to-$m$ mapping. 

    This means that the generating function of \emph{pre-}Q-trees that satisfy the conditions corresponding to $\Delta$-type rigid quadrangulations, weighted by the inverse of the degeneracy, is given by:
    \begin{align}
        \hat{\Delta}(t,x,y)=\sum_{n\geq0}\sum_{j=0}^n\sum_{i=0}^n\frac{1}{n+1}\binom{2n-i}{n}\binom{2n-j}{n}x^{i+1}y^{j+1}t^{n+1}.
    \end{align}
    We can use the same decomposition as before to go from pre-Q-trees to Q-trees giving the claimed generating function.
    We know that the decomposition will not cut at the $i$ consecutive left child edges, since they have (after cycling) strictly positive labels.
\end{proof}

\subsection{Counting B- and C-type rigid quadrangulations}
Using the generating function of $\Delta$-type rigid quadrangulations as building block, we can look at the B- and C-type rigid quadrangulations.
To prove Theorem~\ref{thm:D_B_C_gfs}, we need a small combinatorial lemma:
\begin{lemma}\label{lem:mult_sum}
    For all $m\geq0$, we have
    \begin{align}
        \sum_{\substack{n_k\geq0\\1n_1+2n_2+\dots=m}}\prod_{k\geq1}\frac{1}{k^{n_k}(n_k)!}=1.    
    \end{align}
\end{lemma}
\begin{proof}
    The left-hand side is equal to the $x^m$ coefficient of the generating function
    \begin{align}
        \sum_{m\geq0}x^m\left(\sum_{\substack{n_k\geq0\\1n_1+2n_2+\dots=m}}\prod_{k\geq1}\frac{1}{k^{n_k}(n_k)!}\right)
        =\prod_{k\geq1}\left(\sum_{n_k\geq0}\frac{x^{k n_k}}{k^{n_k}(n_k)!}\right)
        =\prod_{k\geq1}\exp(x^k/k)
        =\exp(\sum_{k\geq1}x^k/k)
        =\frac{1}{1-x},
    \end{align}
    which is $1$ for all $m\geq0$.
\end{proof}

\begin{proof}[Proof of Theorem~\ref{thm:D_B_C_gfs}]
    The first part of the theorem is exactly Proposition~\ref{prop:D_counting}. 
    We will only need to prove the generating functions of B- and C-type rigid quadrangulations.

    \begin{figure}
        \centering
        \includegraphics[width=\linewidth]{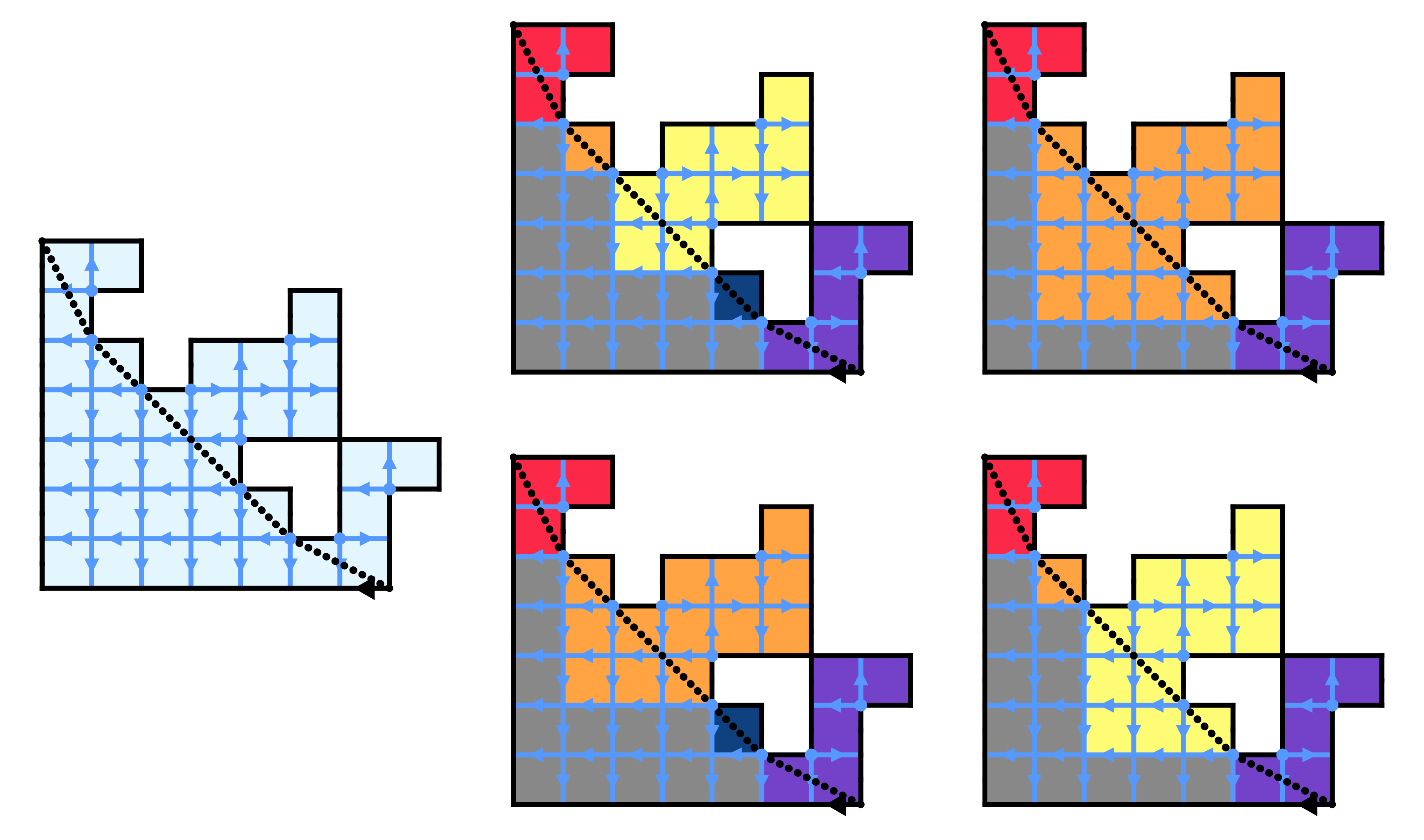}
        \caption{All 4 ways to decompose the B-type rigid quadrangulation on the left into $\Delta$-type rigid quadrangulations. 
        Note that the last decomposition is not considered in the proof of Theorem~\ref{thm:D_B_C_gfs}, since it has a $\Delta$-type rigid quadrangulation of slope~$-1$ and degeneracy~$1$ (in orange), followed by a $\Delta$-type rigid quadrangulation of the same slope, but with higher degeneracy (in yellow).
        The $\Delta$-type rigid quadrangulations have a natural ordering, except $\Delta$-type rigid quadrangulations with the same slope and degeneracy.
        \\
        Also note that the first decomposition is also the decomposition of the B-type rigid quadrangulation into C-type rigid quadrangulations.
        \label{fig:B_disk_decomps}}
    \end{figure}

    A B-type rigid quadrangulation can be naturally decomposed into $\Delta$-type rigid quadrangulations, using the shortest path $P$ the upper corner of the co-base to the right corner of the base, see Figure~\ref{fig:B_disk_decomps}.
    Note that there is a natural ordering when the slope of the path segments is different, but we have to be more careful when there the path has multiple segments with the same slope.
    In this case, there are multiple ways to decompose into $\Delta$-type rigid quadrangulations and the $\Delta$-type rigid quadrangulations can have degeneracies larger than $1$.
    We will show that these two effects exactly cancel.

    Let's consider $m$ segments of the path with the same slope.
    We will only allow partitions of these segments into a sequence of $\Delta$-type rigid quadrangulations, such that the degeneracy of each consecutive $\Delta$-type rigid quadrangulation is non-increasing. 
    We denote the number of $\Delta$-type rigid quadrangulations with degeneracy $k$ by $n_k$.
    Since we weight $\Delta$-type rigid quadrangulations by the inverse of their degeneracy, we get as total weight
    \begin{align}
        \sum_{\substack{n_k\geq0\\1n_1+2n_2+\dots=m}}\prod_{k\geq1}\frac{1}{k^{n_k}(n_k)!}, 
    \end{align}
    where the factorials are to compensate the overcounting of the different orderings of the $\Delta$-type rigid quadrangulations with the same degeneracy. 
    Due to Lemma~\ref{lem:mult_sum}, we get that the weight all these partitions, considering their degeneracies, exactly add up to $1$, so we can consider them effectively as a single $\Delta$-type rigid quadrangulation.

    We conclude that a B-type rigid quadrangulation can effectively be decomposed in a (non-empty) set of $\Delta$-type rigid quadrangulations that have a natural ordering, where the number of convex corners not incident to the (co-)base, and the (co-)base-length is additive. 
    This gives 
    \begin{align}
        B(t,x,y)&=\exp(\Delta(t,x,y))-1,
    \end{align}
    as claimed.
\\
\\
    For the generating function of C-type rigid quadrangulations, we observe that any B-type rigid quadrangulation is naturally decomposable into a non-empty sequence of C-type rigid quadrangulations, see again Figure~\ref{fig:B_disk_decomps}, giving
    \begin{align}
        B=\frac{C}{1-C},
    \end{align}
    which can be inverted to
    \begin{align}
        C=1-\frac{1}{1+B}=1-\exp(-\Delta).
    \end{align} 

    This completes the proof.
\end{proof}

\clearpage
\FloatBarrier
\section{Sampling and outlook}\label{sec:sampling_outlook}
\FloatBarrier
The bijections in this paper allow us to efficiently sample random rigid quadrangulations. 
Sampling some rigid quadrangulations for some large $n$ can give us some ideas of properties that random flat disk might have. 

Furthermore, the tree bijection itself might help to identify the relevant scaling needed, such that random rigid quadrangulations might converge for large $n$ and can even become a tool for proving such convergences or properties that we expect our limiting random object to have.

\subsection{Sampling}
We start by uniformly sampling a pre-Q-tree with base $p$ and size $n_{\textrm{max}}$.
Using the decomposition $\phi$ from Lemma~\ref{lem:decomp}, we get a Q-tree with base $p$ and random size $n$. 
We can use the bijections in this paper to transform this into a rigid quadrangulation.
This will give a random base-$p$ rigid quadrangulation. 
The resulting distribution of $n$ is an open question,\footnote{For $n_{\textrm{max}}\rightarrow\infty$ the distribution seems to converge to some discrete, heavy-tailed distribution.} but when we condition $n$ to be a fixed number, the resulting rigid quadrangulations will have a uniform distribution.  
See Figure~\ref{fig:sim_large_base}.
\\\\
Similarly, one can sample a $\Delta$-type rigid quadrangulation, by sampling a pre-Q-tree with the correct tree structure, cycling the subtrees,\footnote{When multiple cycles are possible, pick one uniformly} using the decomposition to get a Q-tree and then use the bijections to get a $\Delta$-type rigid quadrangulation.

Again, the distribution of $n$ is open, but conditioned on a fixed $n$, the result is a random $\Delta$-type rigid quadrangulation, distributed inversely proportional to its degeneracy.
See Figure~\ref{fig:sim_large_double_base}.
\begin{figure}
    \centering
    \includegraphics[width=\linewidth]{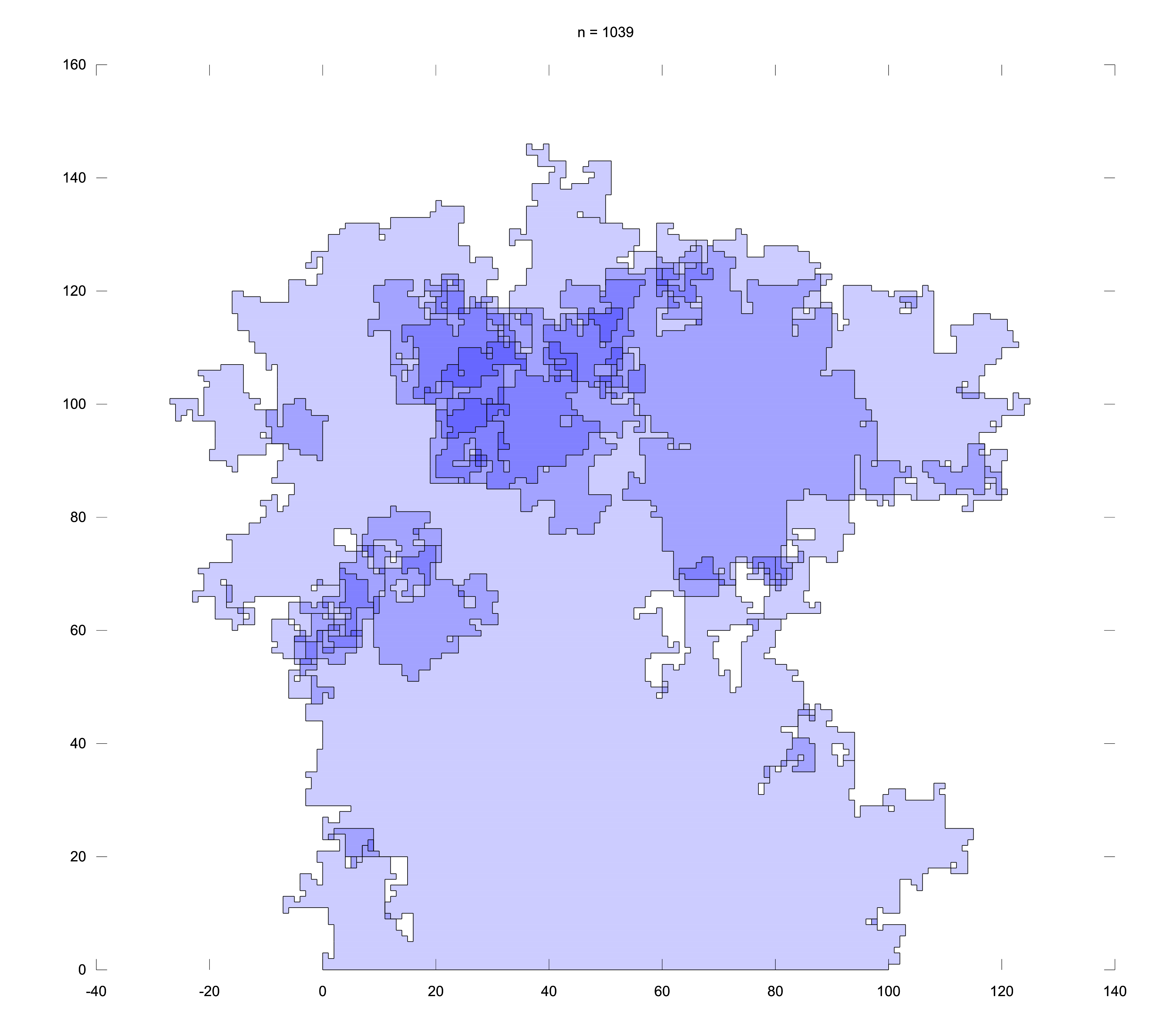}
    \caption{A sample of a uniform random rigid quadrangulation with $n=1039$ non-base convex corners and base-length $p=100$, immersed into the square grid.
    Darker colors correspond to higher number of overlapping parts in the immersion.
    \label{fig:sim_large_base}}
\end{figure}
\begin{figure}
    \centering
    \includegraphics[width=\linewidth]{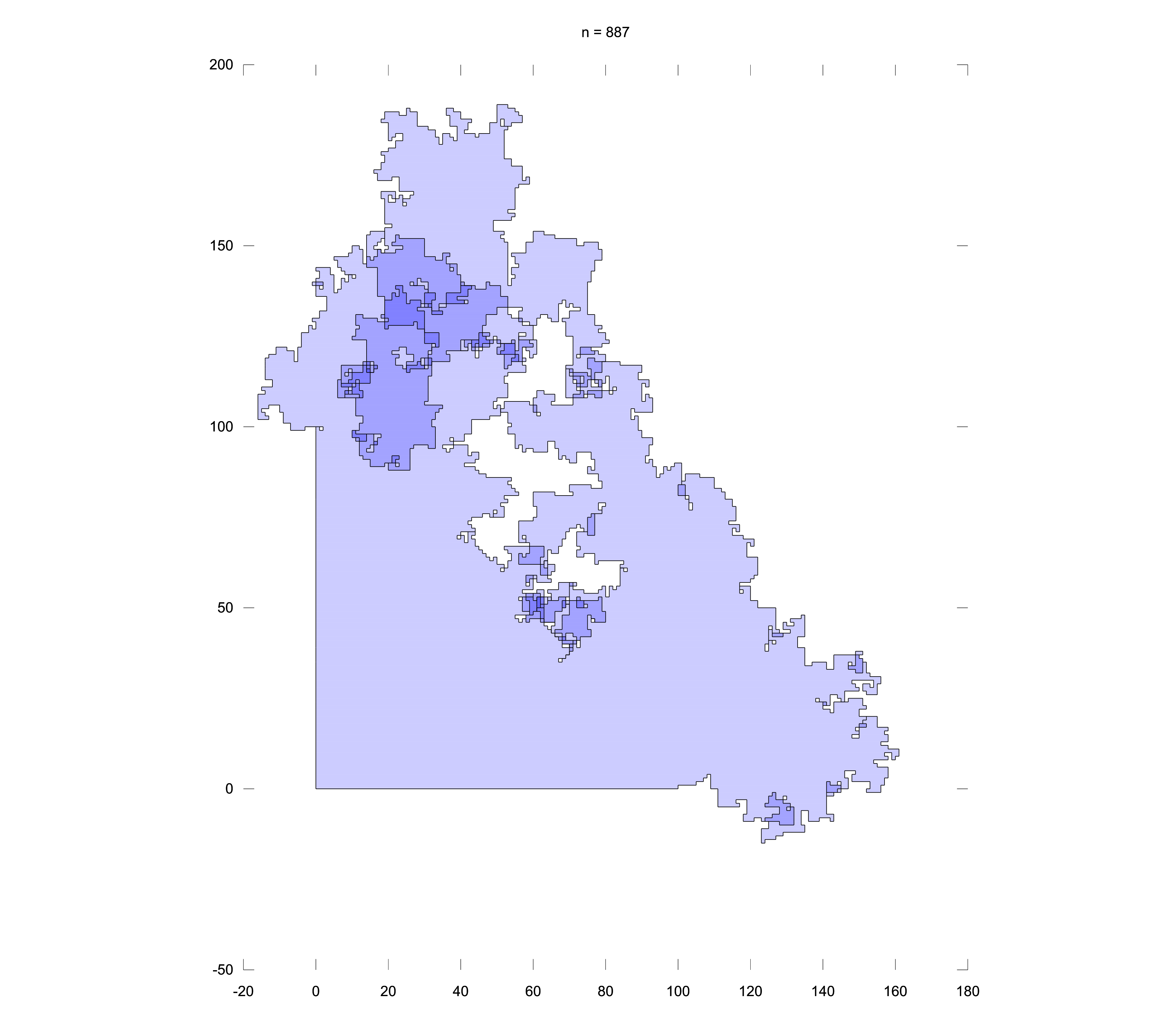}
    \caption{A sample of a uniform random $\Delta$-type rigid quadrangulation with $n=887$ non-base, non-co-base convex corners, base-length $p=100$ and co-base-length $q=100$, immersed into the square grid.
    Darker colors correspond to higher number of overlapping parts in the immersion.
    \label{fig:sim_large_double_base}}
\end{figure}

\subsection{Outlook}
\subsubsection*{Properties of the disk}
Although reproducing the generating functions in a bijective way is nice, we hope to get more insight into random rigid quadrangulations using the tree bijection. 
Some properties of the rigid quadrangulations, like diameter, winding statistics, and area have some (natural) counterparts in the H-trees and Q-trees.
Analyzing the properties of H- and Q-trees would thus be a natural way to study random rigid quadrangulations.

Furthermore, the duality between rigid quadrangulations of the disk and colorful quadrangulations of the sphere, opens the possibility to study random colorful quadrangulations using random \mbox{H-(/Q-)trees,} although it appears to be challenging to find properties of colorful quadrangulations that have a natural counterpart in the trees.

\subsubsection*{\texorpdfstring{$k$}{k}-fold base}
As noted before, for the base, we only consider a special case of Budd's \emph{$k$-fold base} in \cite{Budd_rectilinear_2025}. 
A natural question would be whether the tree bijection of this paper can be naturally extended to the rigid quadrangulations with a $k$-fold base for $k>1$.

\subsubsection*{Comparison to the suggested tree in \cite{Bousquet-Melou_generating_2020,Bousquetmelou2025refinedenumerationplanareulerian}}

In \cite{Bousquet-Melou_generating_2020}, later refined in \cite{Bousquetmelou2025refinedenumerationplanareulerian}, the authors note that the series $t-R(t)$ can be interpreted as the generating function of balanced binary trees with no balanced proper subtree.
Similarly, using the bijection in this paper, we can interpret this series as generating series of base-$0$ pre-Q-trees, where we removed \emph{proper} subtrees using the decomposition map similar to Lemma~\ref{lem:decomp}, which correspond to the removal of base-$0$ pre-Q-trees.

The similarities between the families of trees are clear: they are both binary trees, with a priori $\binom{2n}{n}$ possibilities for a decoration, which are in some sense balanced as a whole, but have balanced proper subtrees removed. 
It should therefore not be surprising that these trees have the same counting sequence.
Yet, finding a direct bijection between the two families is still an open question.

Finding this connection would benefit both trees. 
For the balanced trees, this would give an explicit bijection to colorful quadrangulations, although constructed with a bit of a detour. 
Furthermore, the (pre-)Q-trees are naturally extended to having a higher base, which interpretation is currently lacking of the balanced tree side.

On the other hand, in \cite{Bousquetmelou2025refinedenumerationplanareulerian} it is shown that number of local minima in colorful quadrangulations (which has a natural counterpart in rigid quadrangulations due to \cite{Budd_rectilinear_2025}), corresponds to a natural property of the balanced trees. 
Finding a relation between balanced trees and (pre-)Q-trees could lift this property, which currently cannot be achieved directly.

\subsubsection*{Large \texorpdfstring{$n$}{n}}

Especially for large number of convex corners $n$, the uniform random rigid quadrangulations are expected to converge to some universal random flat disk \cite{Ferrari_Random_Disk_overview_2025,Ferrari_Random_Disk_lattice_2024}. 
In \cite{Budd_rectilinear_2025} some asymptotics are computed which can help to identify this limit.
It is therefore natural to also study the H- and Q-trees in this limit. 
One can hope that the uniform random H-(/Q-)tree will converge to a decorated CRT, and/or that the limit of a random H-tree can be formulated in the language of a self-similar Markov tree \cite{Bertoin_SSMT_2024}, but these questions are still open.

\clearpage
\bibliographystyle{siam}
\bibliography{rectidisk}

\end{document}